\newif\ifsidma

\ifsidma
\documentclass{siamart220329}
\else
\documentclass{amsart}
\fi

\usepackage{tikz}
\usetikzlibrary{decorations.markings}

\ifsidma
\usepackage{stix}
\else
\usepackage{float}
\usepackage{url}
\fi

\usepackage{longtable}
\usepackage{xspace}
\usepackage[e]{esvect}
\usepackage{amsmath}
\usepackage{amssymb}
\usepackage{subcaption}
\usepackage{enumerate}
\usetikzlibrary{patterns}
\usepackage{verbatim}
\usetikzlibrary {positioning}
  \usetikzlibrary{arrows,topaths}
   \usetikzlibrary{calc,intersections}
\usepackage{graphicx}
\usetikzlibrary{positioning, fadings, backgrounds}

\ifsidma\else
\newtheorem{theorem}{Theorem}[section]
\newtheorem{lemma}[theorem]{Lemma}
\newtheorem{proposition}[theorem]{Proposition}
\newtheorem{corollary}[theorem]{Corollary}
\fi

\newcommand{\red}{\xspace\mbox{\textcolor{red}{$\pmb\rightarrow$}}\xspace}
\newcommand{\blue}{\xspace\mbox{\textcolor{cyan}{$\pmb\rightarrow$}}\xspace}
\newcommand{\black}{\xspace\mbox{\textcolor{black}{$\pmb\rightarrow$}}\xspace}
\newcommand{\new}{\xspace\mbox{\textcolor{black}{$\pmb\rightarrow$}}\xspace}
\makeatletter
\newcommand\notsotiny{\@setfontsize\notsotiny{8}{8}}
\makeatother
\newcommand\Tstrut{\rule{0pt}{2.6ex}}         
\newcommand\Bstrut{\rule[-0.9ex]{0pt}{0pt}}   

\tikzstyle{X} = [circle,draw=black!,minimum width=2pt, inner sep=2pt]
\tikzstyle{Xr} = [circle, draw=black!,minimum width=1.5pt, inner sep=1.5pt,fill=red!50]
\tikzstyle{A} = [circle,fill=black!,minimum width=.1em, inner sep=.1em]
 \tikzstyle{B} = [circle,fill=cyan,draw=black!,minimum width=.1em, inner sep=.2em]
 \tikzstyle{W} = [circle,fill=red!,draw=black!,minimum width=.1em, inner sep=.2em]
 \tikzstyle{G} = [circle,fill=green!,draw=black!,minimum width=.1em, inner sep=.2em]

\newcommand{\commentTikz}[1]{} 

 \newcommand{\FaceGraph}{key graph\xspace}
 
 \newcommand{\CanonicalVector}{\mbox{$\boldsymbol{\rho^*}$}}
 \newcommand{\LengthCone}{\mbox{$Q_P^*$}\xspace}
 
 \newcommand{\GapSet}{\mathcal{G}}
 \newcommand{\LeftBorderlineSet}{\mathcal{L}}
 \newcommand{\RightBorderlineSet}{\mathcal{R}}

\DeclareMathOperator{\Conv}{Conv}
 
\title{The length polyhedron of an interval order}

\ifsidma

\author{Csaba Bir\'o\thanks{Department of Mathematics, University of Louisville, Louisville, Kentucky, U.S.A.
               (\email{csaba.biro@louisville.edu}, \email{kezdy@louisville.edu}).}
       \and Andr\'{e} E. K\'{e}zdy\footnotemark[1]
       \and Jen\H{o} Lehel\footnotemark[1] \thanks{Alfr\'ed R\'enyi Institute of Mathematics, Budapest, Hungary 
               \email{lehelj@renyi.hu}).}}

\else

\author{Csaba Bir\'o}
\address{Department of Mathematics, University of Louisville, Louisville, Kentucky, U.S.A.}
\email{csaba.biro@louisville.edu}

\author{Andr\'{e} E. K\'{e}zdy}
\address{Department of Mathematics, University of Louisville, Louisville, Kentucky, U.S.A.}
\email{kezdy@louisville.edu}

\author{Jen\H{o} Lehel}
\address{Alfr\'ed R\'enyi Institute of Mathematics, Budapest, Hungary} 
\email{lehelj@renyi.hu}

\fi

\begin{document}

\ifsidma\maketitle\fi

\begin{abstract}
The length polyhedron $Q_P$ of an interval order $P$ is the convex hull of integral vectors representing the
interval lengths in interval representations of $P$. This polyhedron has been studied by
various authors, including Fishburn and Isaak. Notably, $Q_P$ forms a pointed
affine cone, a property inherited from being a projection of the representation polyhedron, a
structure explored also by Doignon and Pauwels. The apex of the length polyhedron corresponds
to Greenough's minimal endpoint representation, which is, in fact, the length vector of the
canonical interval representation---an interval representation that minimizes the sum of the interval
lengths.

Building on a combinatorial perspective of canonical representations, we refine Isaak's graph-theoretical model by introducing a new and simpler directed graph. From directed cycles of this
key digraph, we extract a linear system of inequalities that precisely characterizes the length
polyhedron $Q_P$. This combinatorial approach also reveals the unique Hilbert basis of the polyhedron. We prove that the intersection graph of the sets
corresponding to these binary rays are Berge graphs; therefore they are perfect graphs. As a
result, for interval orders with bounded width, the length polyhedron has a polynomial-sized Hilbert
basis, which can be computed in polynomial time. We also provide an example of interval orders with a Hilbert basis of exponential size. In a
companion paper we determine the Schrijver system for the length polyhedron. We conclude with
open problems. 
  \hfill \textcolor{cyan}{/\today/}
\end{abstract}

\ifsidma\begin{keywords}\else\keywords{\fi
interval order, minimal endpoint representations, key graph model, conic polyhedron, Hilbert basis, perfect graphs%
\ifsidma\end{keywords}\else}\fi

\ifsidma\begin{MSCcodes}\else\subjclass{\fi
06A06, 90C27, 05C62, 52B05, 05C20%
\ifsidma\end{MSCcodes}\else}\fi

\ifsidma\else\maketitle\fi

 \section{Introduction}
 
In this paper we investigate  the interval length polyhedron whose origin is implicit in Fishburn's work on  interval orders  \cite{FishburnBook}. 
A partially ordered set  $P=(X,\prec)$ is an {\it interval order} if there is an assignment of compact intervals
from $\mathbb{R}$ to elements of $X$, say $I_x = [\ell_x, r_x]\subset \mathbb{R}$ is assigned to $x \in X$, such that $x\prec y$ if and only if $r_x < \ell_y$ in $\mathbb{R}$. 
Moreover, we require  $r_x+1 \leq \ell_y$, for all pairs $x\prec y$, an assumption that makes possible to focus on  {\it natural representations}
with  integral interval endpoints.  
For $X=\{1,\ldots,n\}$,  the $2n$ variables $\{\ell_i,r_i\}_{i=1}^n$  define
the $\rho$-vector $(\rho_1,\ldots,\rho_n)\in\mathbb{R}^n$, where $\rho_i=r_i-\ell_i$, $1\leq i\leq n$. The existence of an interval representation of $P$ with a given $\rho$-vector can be stated as the feasibility of an appropriate  system of inequalities on the reals.
The interval {\it length polyhedron} $Q_P\subset \mathbb{R}^n$ is defined to be the set of all $\rho$-vectors for which this system of inequalities is feasible. 

Fishburn gave the definition of  a system of length inequalities of the form $$\sum\limits_{a\in A}\rho_a\ <\ \sum\limits_{b\in B}\rho_b,$$ where $A,B\subset X$ are disjoint non-empty sets satisfying the property, that these inequalities
{\it need} to be true in {every} interval representation of $P$. He proved \cite[Theorem 2, p. 126]{FishburnBook},  if a vector $\rho\in \mathbb{R}^n$ satisfies the length inequalities, then there is an interval representation of $P$ where the interval lengths are equal to the components of  $\rho$. Fishburn also noticed that the feasibility of this system of inequalities on real numbers is equivalent with the feasibility on integers.  
Isaak \cite{Isaak1993, Isaak2009} simplified Fishburn's length inequalities  and developed an infeasibility test in terms of a weighted bipartite directed graph model.  

In 2011 Doignon and Pauwels \cite{DoignonPauwels} considered a related polyhedron associated with interval orders  the {\it 
location polyhedron} of the interval order $P$. 
 {Their location polyhedron, $R_\epsilon^P$, where $\epsilon>0$ is fixed,  
consists of  all vectors $(\ell_1,r_1,\ldots,\ell_n,r_n)\in \mathbb{R}^{2n}$
 for which 
 $[\ell_i,r_i]_{i=1}^n$ is a real interval representation of $P$  
 with the additional condition that $r_x+\epsilon\leq \ell_y$, for every pair $x\prec y$.}  
 They proved that 
the  location polyhedron $D_P=R_1^P$  is a pointed affine cone and, although they did not realize it, its apex\footnote{\ The vertex of a pointed cone is called here {\it apex}.} 
corresponds  to Greenough's minimal endpoint representation  \cite[Theorem 2.4]{Greenough}.

In our study we introduce key graphs by applying Fourier-Motzkin Elimination to the linear system defining the location polyhedron. The key graph is  a colored directed graph model of the interval order $P$  inspired by Greenough's minimal endpoint representation of $P$, and different from the one Isaac used \cite{Isaak1993}.
The key graph encodes a linear system defining the length polyhedron $Q_P$.
We show that the length polyhedron $Q_P$ is a pointed affine cone  with its apex at the length vector of the canonical representation of $P$. In  Section \ref{Hilbertsection} we investigate the V-description  of the cone $Q_P$ giving the Hilbert basis. In a companion paper \cite{KeLe} we consider the H-description of the length polyhedron and determine the unique Schrijver system.

Our motivation to study this polyhedron $Q_P$ arises from an attempt to attack $k$-count problems, questions about interval representations using at most $k$ distinct interval lengths.
A systematic study of the $k$-count representations for interval orders 
was initiated by Fishburn a half a century ago \cite{FishburnBook}. Since then several remarkable results were obtained by Isaak \cite{Isaak1993,Isaak2009} with several collaborators 
\cite{BoyadzhiyskaDissertation,Gordon,BGT}; and also by other groups with  
Oliviera and  Szwarcfiter \cite{FMOS,JLORS}. 
Nevertheless, 
the $k$-count problems are wide open in general and remain a challenge even for $k=2$. Investigating the polyhedra attached to interval orders, in particular, focusing on the length polyhedron, may help understand the 
difficulty of the 2-count problems. Related problems are addressed in Section \ref{Conclusionsection}; some progress is reported in this direction separately \cite{BiKeLe}. Dimension bounds for posets of bounded interval count are studied in \cite{BiroWan}.

Interval orders have a rich history and have been applied to a wide range of fields,
including scheduling, resource allocation, temporal reasoning (e.g., language processing and
artificial intelligence),
time-series analysis (e.g., event prediction, pattern detection, segmentation), psychology
(e.g., preference modeling, decision making), bioinformatics (e.g., protein interaction networks,
gene expression analysis),
economic modeling (e.g., utility, market competition, consumer behavior), data visualization, and
software engineering (e.g., version control).

Early work in mathematical psychology, particularly on Guttman scales, led Ducamp and Falmagne
\cite{DuFa}  to
explore related relations, which were later named biorders by Doignon, Ducamp, and Falmagne
\cite{DoDuFa}.
Interval orders, a special class of biorders, were characterized by Fishburn \cite{FishburnBook}. Doignon and
Pauwel's research \cite{DoignonPauwels} on the
representation polyhedron of interval orders,  discussed later, stemmed in part from mathematical
psychology
questions related to modeling binary preferences in knowledge spaces, a key area of knowledge
space theory.
This theory was originally developed by Doignon and Falmagne \cite{DoFa}.

A significant modern application of this work is the ALEKS system, a personalized learning platform
for
students co-founded by Jean-Claude Falmagne and Nicolas Thiery in 1997. ALEKS leverages
learning spaces,
a type of knowledge space that models an individual student's domain knowledge as a
combinatorial structure.
 \subsection{Overview}
Doignon and Pauwels gave an elegant H-description \cite[Theorem 4]{DoignonPauwels} and a V-description  [Ibid.,Theorem 6] of their location polyhedron $D_P$. 
The length polyhedron $Q_P$ and their location polyhedron $D_P$  
are derived by linear transformations and projections 
from the full representation polyhedron $R_P\subset\mathbb{R}^{3n}$
 defined by  the feasible vectors $(\ell_1,r_1,\ldots,\ell_n,r_n,\rho_1,\ldots,\rho_n)$.
The length polyhedron  is the result of a linear transformation of the location domain in $\mathbb{R}^{2n}$ followed by a projection into the length domain in $\mathbb{R}^{n}$. 
 In Section \ref{Fullsection} we present the full polyhedron $R_P$ in details.  We give alternate combinatorial proofs of  Doignon and Pauwels results concerning the location polyhedron $D_P$  in Section \ref{Locsection}.

Our characterization of the length polyhedron $Q_P$  is more complicated, and
improves Fishburn's length inequalities \cite[p.126]{FishburnBook}. In Section \ref{Lengthsection}  we show that
these inequalities can be extracted from directed cycles of the key graph. The key graph could also be defined in terms of noses and hollows, concepts introduced 
 by Pirlot \cite{Pirlot} while investigating reduced semiorders.
 These concepts were  used by Doignon and Pauwels \cite{DoignonPauwels} in describing the location polyhedra of interval orders. 

In Section \ref{Hilbertsection} the unique Hilbert basis of the cone $Q_P$  is determined, and we give examples to show that exponentially many facet-defining inequalities may appear, though only polynomial many are possible for bounded width interval orders. For all of our examples we apply Johnson's algorithm \cite{Johnson}, which finds all directed cycles in a directed graph.

Section \ref{Canonicalsection} introduces the canonical representation  whose fundamental role among all representations of an interval order form the starting point of our investigations. We use different techniques from linear algebra, combinatorial optimization, and graph theory.  In the next subsection we gather the necessary definitions and basic results that we will use in the paper. We conclude with open problems in Section \ref{Conclusionsection}.
 
\subsection{Preliminaries}
Ducamp and Falmagne \cite{DuFa} (see also Fishburn \cite{Fish2+2})  characterized interval orders as partial orders with no {\bf 2 + 2} as a subposet, that has four elements, $\{a, b, c, d\}$  
and two relations, $a\prec b$ and $c\prec d$.   
An interval representation of an interval order with integral endpoints is called here an {\it integral representation};  an integral representation with non-negative integral endpoints is referred as to a {\it natural representation}. 

The forbidden subposet characterization  of an  interval order $P=(X,\prec)$ implies that the down-sets 
$D(z)=\{x\in X : x\prec z\}$ are ordered by inclusion, and the same is true for
 the up-sets $U(z)=\{x\in X : z\prec x\}$, $z\in X$. Greenough proved \cite[Theorem 2.3]{Greenough} that the common number $m$ of the distinct down-sets and up-sets, called the {\it magnitude} of $P$, gives the smallest number of distinct endpoints of an interval representation of $P$. The {\it canonical representation} of $P$ is defined here as a natural representation with endpoints $0,1,\ldots,m-1$, where $m$ is the magnitude of $P$. The linear orders in our examples are distinguished by labels  using their ascent sequence identification, ASI, due to Bousquet et al. \cite{ascent2010}.  
 For further terms and definitions  in poset theory see Trotter \cite{TrotterBook}. 

The collection of the length vectors of an interval representation for a fixed interval order $P$ forms a convex polyhedron. A linear system of inequalities defining a polyhedron has several equivalent forms depending on the variables we keep in the system.
There are different polyhedra in Euclidean spaces of different dimension; one polyhedron is obtained from another by linear transformations and projections. An obvious algorithm to perform a projection by manipulating the defining inequalities is the Fourier-Motzkin Elimination procedure explained and used in Section \ref{Lengthsection}. 

{A (rational) polyhedron is called an {\em integral polyhedron} if it is the convex hull of its integral points. Alternatively,  every non-empty face of an integral polyhedron contains an integral point.}  The various polytopes attached to interval representations
are integral polyhedra.  Total unimodularity  
of  the coefficient matrices of the defining linear systems plays an  important role in this property. A matrix is {\em totally unimodular} if all of its square submatrices have a determinant equal to $0$, $1$, or $-1$.
 A $\{0,\pm 1\}$-matrix $A$ has an {\em equitable bicoloring} if its columns can be partitioned into black columns and white columns
so that, for every row of $A$, the sum of the entries in the black columns differs from the sum of the entries in the white columns by at most one.   
\begin{theorem} [Ghouila-Houri \cite{G-H} (1962)]
 \label{Ghouila-Houri} 
 A $\{0,\pm 1\}$-matrix $A$ is totally unimodular if and only if
	every submatrix of $A$ has an equitable bicoloring.
\end{theorem}
\begin{theorem} [Hoffman and Kruskal \cite{HoffmanKruskal}(1956)] 
\label{HoffmanKruskal} 
{\it An integral matrix $A$ is totally unimodular if and only if for each integral vector $b$ the polyhedron $\{x : x\geq 0, Ax\leq b\}$ is integral.}
\end{theorem}
For further definitions and results in optimization and integer programming see Schrijver \cite{Schrijver}.
 
 Besides linear algebra techniques we use graph theory models, and fundamental  graph theory results in algorithmic consideration. 
{A graph is {\em perfect} if for all induced subgraphs $H$ the chromatic number of $H$ equals the size of a maximum clique in $H$.} Several hard combinatorial problems can be stated in terms of graphs, and they have polynomial solution when restricted to the family of perfect graphs.  

 In the study of perfect graphs an induced cycle is called a hole, and its complement is called an antihole. 
 A particular family,  called Berge graphs, are those graphs that contain no odd hole and no odd antihole. The strong perfect graph theorem due to Chudnovsky et al.  \cite{Chudnovsky} says that the family of perfect graphs is identical with the family of Berge graphs.

For further definitions and results in graph theory see West \cite{DougBook}

\section{Minimal representations of an interval order}
\label{Canonicalsection}

Interval orders have appeared in different contexts and various guises  
such as representations due to Bogart \cite{Bogart1993},  Greenough \cite{Greenough},  Doignon \cite{Doignon1988}, the  polyhedral  representation due to Doignon and Pauwels \cite{DoignonPauwels}, or the ascent sequences by Bousquet-M\'elou et al. \cite{ascent2010}.  One particular interval representation that we call  the canonical representation  
reappears as the optimum for a few rather distinct objectives.  In this section we present a few different occurrences and the basic properties of these canonical representations, which we later relate to the length polyhedron  
of an interval order.  

The best known representation of an interval order $P=(X,\prec)$ is due to Greenough \cite{Greenough},  that we call  {\it magnitude representation},  minimizes the number of distinct endpoints.  
 Greenough proved  \cite[Theorem 2.4]{Greenough}  that the {magnitude  representation} of $P$ is essentially unique, i.e., when we disregard the physical locations, the endpoint incidences agree in any minimal endpoint representation.

Doignon \cite{Doignon1988,DoignonPauwels}  introduced the {\it ubiquitously  minimal representation} 
that simultaneously minimizes all  endpoints in a natural representation of an interval order. 
He proved by applying the existence of minimal potentials in cycle free directed graphs that 
 there is a natural representation of $P$, $\{[\ell_x^*,r_x^*] : x\in X\}$, such that
 $\ell_x^*\leq \ell_x$ and $r_x^*\leq r_x$,  for every natural representation $\{[\ell_x,r_x] : x\in X\}$. In particular, an ubiquitously  minimal representation is a magnitude representation  
 as described by Greenough \cite{Greenough}. 
 In Proposition \ref{minCanon}
we show  a direct argument that the ubiquitously minimal representation is identical with Greenough's canonical representation. 

The {\it canonical representation} of an interval order unifies all properties, concepts and structures in our investigations. 
In particular, the canonical representation minimizes the maximum of the right endpoints in a natural representation of a given interval order. 
  Greenough's result  \cite[Theorem 2.4]{Greenough} on the  minimality and uniqueness of the magnitude representation  yields 
  a combinatorial characterization of a canonical  representation as follows.
  
  \begin{lemma} [Greenough \cite{Greenough}]
 \label{canon}
 A family $\mathcal{R}$ of intervals with  interval endpoints in the set  $\{0,1,\ldots,m-1\}$ is the canonical 
 representation of some interval order of magnitude $m$ if and only if every $i$, $0\leq i<m$,  is a  left endpoint of some interval and also the  right endpoint of some interval. 
 \end{lemma}
 If the conditions in Lemma \ref{canon} are not satisfied by a natural representation $\mathcal{R}$ of $P$, then $\mathcal{R}$  is not minimal in terms of the number of interval endpoints. In this case the number of distinct endpoints can be reduced by collapsing, which is explained next.  
 
 An integral interval $[i,j]$  is called a {\it gap} in a natural representation if it contains no interval endpoint in its interior. 
{\it Collapsing a gap} is the operation of adjusting a representation as expressed in the next lemma. 

   \begin{lemma} 
 \label{collapse}
 Let $\mathcal{R}=\{[\ell_x,r_x] : x\in X\}$  be a natural representation of the interval order $P=(X,\prec)$, and let $i<j$ be consecutive  interval endpoints in $\mathcal{R}$.
If 
   $i$ is not a right endpoint,
  then  for every $x\in X$   replace  $q\in \{\ell_x,r_x\}$, $q\geq j$,   with new interval endpoint $q^\prime=q-(j-i)$.   
  If    $j$ is not a left endpoint,
  then for every $x\in X$
    replace $q\in \{\ell_x,r_x\}$, $q\leq i$,   with new interval endpoint $q^\prime=q+(j-i)$.       
  In each case the  family $\mathcal{R}^\prime$ of the modified intervals represent $P$ with fewer endpoints; furthermore, interval lengths do not increase.
 \end{lemma} 
 \begin{proof} 
    Fig.\ref{ExtendingCollapsing}  
 	shows how $\mathcal{R}$ is collapsed to a new interval representation $\mathcal{R}^\prime$ of $P$, 
 	essentially by contracting the interval $[i,j]$ to $[i]$.  Intersecting intervals remain intersecting; furthermore, disjoint intervals remain disjoint,
 since a new left endpoint $q^\prime=i$ does not coincide with any old right endpoint on line $i$. 
 Intervals of type $B, C, D, E$ contain the gap $[i,j]$, hence their length decreases by $j-i$; 
 the length of intervals of type A, F, G, H does not change.
 
  Similar argument applies for collapsing when     $j$ is not a left endpoint of any interval.
 \end{proof}
 
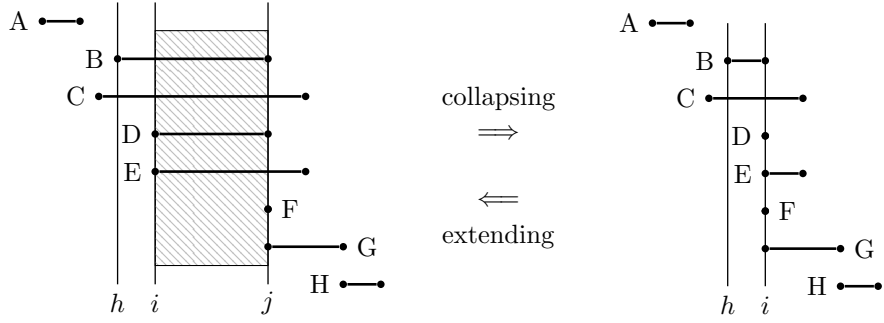
\begin{figure}[htp]
\vskip.05em
	\def\HeightOfBars{7}  
	\begin{center}
		\begin{tikzpicture}[scale=.5]
   \draw[line width=.05em](3,0)--(3,.5+\HeightOfBars);\node()at(3,-0.5) {$i$};
   \draw[pattern=north west lines, pattern color=gray!50] (3,0.5) rectangle (6,-.25+\HeightOfBars);
   \draw[line width=.05em](6,0)--(6,.5+\HeightOfBars);\node()at(6,-0.5) {$j$};
\draw[line width=.05em](2,0)--(2,.5+\HeightOfBars);\node()at(2,-0.5) {$h$};
	\node[A,label=left:A](LA) at (0,7) {}; \node[A](RA) at (1,7) {};
	\draw[line width=.1em](LA)--(RA);
	
	\node[A,label=left:B](LB) at (2,6) {}; \node[A](RB) at (6,6) {};
	\draw[line width=.1em](LB)--(RB);
	
   \node[A,label=left:C](LC) at (1.5,5) {}; \node[A](RC) at (7,5) {};
   \draw[line width=.1em](LC)--(RC);
   
   \node[A,label=left:D](LD) at (3,4) {}; \node[A](RD) at (6,4) {};
   \draw[line width=.1em](LD)--(RD);
   
   \node[A,label=left:E](LE) at (3,3) {}; \node[A](RE) at (7,3) {};
   \draw[line width=.1em](LE)--(RE);
   
   \node[A,label=right:F](LF) at (6,2) {}; \node[A](RF) at (6,2) {};
   \draw[line width=.1em](LF)--(RF);
   
      \node[A](LG) at (6,1) {}; \node[A,label=right:G](RG) at (8,1) {};
   \draw[line width=.1em](LG)--(RG);
   
      \node[A,label=left:H](LH) at (8,0) {}; \node[A](RH) at (9,0) {};
   \draw[line width=.1em](LH)--(RH);
	
\end{tikzpicture}
		\hskip 0.5cm
		\begin{tikzpicture}[scale=.3]
			\node[](S) at (0,0) {$\quad$};
			\node[label=above:collapsing](C) at (0,8) {$\Longrightarrow$};
			\node[label=below:extending](E) at (0,5) {$\Longleftarrow$};
		\end{tikzpicture}	
		\hskip0.5cm
\begin{tikzpicture}[scale=.5] 
\draw[line width=.05em](3,0)--(3,\HeightOfBars);\node()at(3,-0.5) {$i$};
\draw[line width=.05em](2,0)--(2,\HeightOfBars);\node()at(2,-0.5) {$h$};

\node[A,label=left:A](LA) at (0,7) {}; \node[A](RA) at (1,7) {};
\draw[line width=.1em](LA)--(RA);

\node[A,label=left:B](LB) at (2,6) {}; \node[A](RB) at (3,6) {};
\draw[line width=.1em](LB)--(RB);

\node[A,label=left:C](LC) at (1.5,5) {}; \node[A](RC) at (4,5) {};
\draw[line width=.1em](LC)--(RC);

\node[A,label=left:D](LD) at (3,4) {}; \node[A](RD) at (3,4) {};
\draw[line width=.1em](LD)--(RD);

\node[A,label=left:E](LE) at (3,3) {}; \node[A](RE) at (4,3) {};
\draw[line width=.1em](LE)--(RE);

   \node[A,label=right:F](LF) at (3,2) {}; \node[A](RF) at (3,2) {};
\draw[line width=.1em](LF)--(RF);

\node[A](LG) at (3,1) {}; \node[A,label=right:G](RG) at (5,1) {};
\draw[line width=.1em](LG)--(RG);

\node[A,label=left:H](LH) at (5,0) {}; \node[A](RH) at (6,0) {};
\draw[line width=.1em](LH)--(RH);			
		\end{tikzpicture}
	\end{center}	
	\caption{Collapsing an interval $[i,j]$} 
	 \label{ExtendingCollapsing}
\end{figure}	

\begin{proposition} 
\label{compressing} Let  $\mathcal{C}$ be the canonical representation of $P$, and let $\mathcal{R}$ be an arbitrary natural representation of the same interval order.
Then $\mathcal{R}$ can be transformed into $\mathcal{C}$ by a sequence of collapsing operations. 
\end{proposition}
\begin{proof}
 Repeated application of Lemma \ref{collapse} leads from $\mathcal{R}$ to a representation which is the canonical representation of $P$, by Lemma \ref{canon}.
\end{proof}
 Notice that a sequence of the obvious reverse transformations, so called {\it gap extensions},
   take   the canonical representation  of an interval order to any natural representation.
  
The fundamental role of the canonical representation $\mathcal{C}=\{[\ell_x^*,r_x^*] : x\in X\}$ among all integral representations of a fixed interval order $P=(X,\prec)$  is expressed in the next auxiliary lemma.  

  \begin{lemma} 
 \label{right}
 If   $\mathcal{R}=\{[\ell_x,r_x] : x\in X\}$ is a natural representation of $P=(X,\prec)$  and $a,b\in X$ are such that 
 $r_a^*<\ell_b^*$,
 then  $\ell_b^*-r_a^* \leq  \ell_b-r_a$.
  \end{lemma}
 \begin{proof} Suppose that $\ell_b^*-r_a^*=k$, for some $k\geq 1$.  By Lemma \ref{canon}, there are intervals 
 $x_i,y_i \in X$, $i=0,1,\ldots,k$, such that $r_{x_i}^*=\ell_{y_i}^*=r_a^*+i$; note that $x_0=a$ and  $y_k=b$.

 Then we have $x_i\| y_{i}$, that is $\ell_{y_i}\leq r_{x_i}$, for $i=0,1\ldots,k$; and $x_i\prec y_{i+1}$, that is  $r_{x_i}<  \ell_{y_{i+1}}$ for $i=0,1\ldots,k-1$. Therefore, 
$$r_a=r_{x_0}<\ell_{y_1}\leq r_{x_1}<\ell_{y_2}\leq r_{x_2} < \ldots < \ell_{y_{k-1}}\leq r_{x_{k-1}}< \ell_{y_k}=\ell_b.$$
Observe that there are $k$ strict inequalities between $r_a$ and $\ell_b$ in the sequence of inequalities above, 
 thus $\ell_b- r_a\geq k= \ell_b^*- r_a^*$ follows.
 \end{proof}
 
 An immediate application of Lemma \ref{right} yields that
the ubiquitously minimal representation due to Doignon \cite{Doignon1988,DoignonPauwels} is identical with the canonical representation.

 \begin{proposition}
 \label{minCanon}
 If $\mathcal{R}=\{[\ell_x,r_x] : x\in X\}$ is a natural representation of $P=(X,\prec)$,
 and $\mathcal{C}=\{[\ell_x^*,r_x^*] : x\in X\}$ is its canonical representation,  
 then  $\ r_x^*\leq r_x \ \text{\rm and\ }\  \ell^*_x\leq \ell_x$, for every $x\in X$.
\end{proposition}
\begin{proof}  The  claim is obvious   for $r^*_x=0$. Assume $ r^*_x>0$. By Lemma \ref{canon}, there are $a,b\in X$ such that $r_a^*=0,$ and $\ell_b^*=r^*_x>0$. Since  $b\| x$ in $P$, $\ell_b\leq r_x$. Then by  applying Lemma \ref{right}, we obtain
$
r^*_x=\ell_b^*= \ell_b^*- r^*_a\leq \ell_b- r_a\leq r_x-r_a\leq r_x.$ To see the ubiquitous minimality for the left endpoints we apply 
 Lemma \ref{right}  with the condition $0=r_a^*<\ell_x^*$ that leads to
$\ell_x^*=\ell_x^*-r_a^*\leq \ell_x-r_a\leq \ell_x.$ 
\end{proof}

\subsection{Slack}
\label{slacksection}
Let $\mathcal{R}=\{[\ell_x,r_x] : x\in X\}$  be an integral representation of $P=(X,\prec)$.
For  $x,y\in X$ the \emph{slack in $\mathcal{R}$}  of the pair $(x,y)$  is defined to be
\[s_{\mathcal{R}}(x,y)= 
\left\{
\begin{array}{lcl}
 \ell_y-r_x-1 &\text{ if}  & x\prec y\\
  r_y-\ell_x & \text{ if}  &    x\|y\\
    \rho_x= r_x-\ell_x& \text{ if}  &    x=y
\end{array} \quad .
\right.
\]
Note that $s_{\mathcal{R}}(x,y)\geq 0$, and it is defined for every pair $(x,y)$, unless $y\prec x$. 

The  next theorem highlights the special role  the canonical representation holds among all  
  integral representations.

\begin{theorem} [Slack Theorem]
\label{slackthm}
Let $\mathcal{C}$ be the canonical representation and let $\mathcal{R}$ be  an arbitrary integral representation of the interval order $P=(X,\prec)$. If
$x,y\in X$ are such that
$y\not\prec x$, then  $s_{\mathcal{C}}(x,y)\leq  s_{\mathcal{R}}(x,y)$.
\end{theorem}
\begin{proof} 
 Let $\mathcal{C}=\{[\ell_x^*,r_x^*] : x\in X\}$ be the canonical representation of $P$, and let $\mathcal{R}=\{[\ell_x,r_x] : x\in X\}$.
 For $x\prec y$, Lemma \ref{right} implies
$$s_{\mathcal{C}}(x,y)=\ell_y^*-r_x^*-1\leq \ell_y-r_x-1=s_{\mathcal{R}}(x,y).$$

If $x\| y$ or $x=y$, then we apply Lemma \ref{canon} to find intervals $x_0,y_0\in X$ such that
$r_{x_0}^*=\ell_x^*$ and $\ell_{y_0}^*=r_y^*$. Then we have  $ \ell_{y_0} \leq r_y$ and
$ r_{x_0}\geq \ell_x$, thus Lemma \ref{right} implies
$$s_{\mathcal{C}}(x,y)=r_y^*-\ell_x^*= \ell_{y_0}^*-r_{x_0}^*\leq\ell_{y_0}-r_{x_0}\leq  r_y-\ell_x=s_{\mathcal{R}}(x,y),$$
and the theorem follows.
\end{proof}

An immediate corollary of Theorem \ref{slackthm} is that if $s_{\mathcal{R}}(x,y)=0$ is true  in some representation $\mathcal{R}$ of $P$, then $s_{\mathcal{C}}(x,y)=0$ in the canonical representation $\mathcal{C}$ as well.
We call  $(x,y)$ a {\it slack zero pair}, provided $y\not\prec x$ and $s_{\mathcal{C}}(x,y)=0$. 
\begin{figure}[htp]
\begin{center}
\begin{tikzpicture}[scale=.8] 

\node[X,label=right:4](1) at (.5,3) {};
\node[X,label=left:6](3) at (-.5,3) {};
\node[X,label=right:2](4) at (2,3.75) {};

\node[X,label=below:3](x) at (2,1.7) {};
\node[X,label=below:1](y) at (0.5,1.7) {};
\node[X,label=above:5](z) at (1,5.7) {};

\node[X,label=above:7](2) at (2,5) {};

\draw[line width=.1em](z)--(3)--(y)--(1) (4)--(2)--(1) (x)--(4)--(y);
\draw[line width=.1em](1)--(z)--(4);
\end{tikzpicture}
\hskip2cm
\begin{tikzpicture}
\foreach \i in {0,...,4}{
    \draw[line width=.01em](1+\i,0)--(1+\i,2);\node()at(1+\i,2.2) {\tiny\i};
}
\node[A,label=left:4](L) at (2,1) {};\node[A](R) at (3,1) {};
\draw[line width=.1em](L)--(R);
\node[A,label=left:7](L) at (4,1) {};\node[A](R) at (5,1) {};
\draw[line width=.1em](L)--(R);
\node[A,label=left:2](L) at (3,.5) {};\node[A](R) at (3,.5) {};
\draw[line width=.1em](L)--(R);
\node[A,label=left:6](L) at (2,1.5) {};\node[A](R) at (4,1.5) {};
\draw[line width=.1em](L)--(R);
\node[A,label=left:3](L) at (1,.5) {};\node[A](R) at (2,.5) {};
\draw[line width=.1em](L)--(R);
\node[A,label=left:1](L) at (1,1) {};\node[A](R) at (1,1) {};
\draw[line width=.1em](L)--(R);
\node[A,label=left:5](L) at (5,1.5) {};\node[A](R) at (5,1.5) {};
\draw[line width=.1em](L)--(R);
\end{tikzpicture}
\end{center}
\caption{Hasse diagram and  the canonical representation of 
the interval order $P=[0,1,0,1,3,1,3]$ that has magnitude $m=5$}
\label{ex0}
\end{figure}
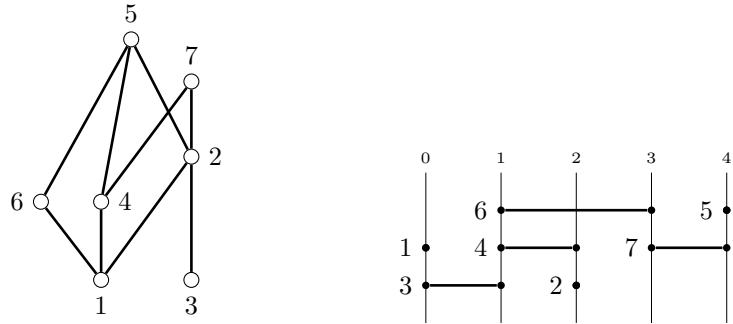

Fig.\ref{ex0} shows the interval order $P=[0,1,0,1,3,1,3]$;  in  this example  $(2,2), (6,5),$ $(6,3), (7,6)$ are slack zero pairs, while,  $(7,7), (2,5), (6,4), (6,7)$ are not slack zero pairs.

\begin{theorem} [Converse Slack Theorem] 
\label{conversethm} Let $\mathcal{C}$  be the canonical representation of the interval order $P=(X,\prec)$, and let ${\mathcal{R}}$ be an integral representation of $P$. 
If  $s_{\mathcal{C}}(x,y)=0$ implies $s_{\mathcal{R}}(x,y)=0$, for every $y\not\prec x$, 
then ${\mathcal{R}}$ is a shift of ${\mathcal{C}}$, that is, the intervals in ${\mathcal{R}}$ are obtained 
 by adding the same integer to every interval endpoint in $\mathcal{C}$. 
\end{theorem}

\begin{proof} 
Assume that two interval endpoints coincide in the canonical representation ${\mathcal{C}}=\{[\ell_x^*,r_x^*] : x\in X\}$. If $\ell^*_x=r^*_y$, 
then $s_{\mathcal{C}}(x,y)=0$; this implies $s_{\mathcal{R}}(x,y)=0$, that is, $\ell_x=r_y$ in ${\mathcal{R}}$. By Lemma \ref{canon}, every  $i$, $0\leq i\leq m-1$, is a left endpoint and a right endpoint  in ${\mathcal{C}}$; therefore, coinciding endpoints in ${\mathcal{C}}$ correspond to coinciding endpoints in ${\mathcal{R}}$. Thus we obtain that $R$ is a minimal endpoint representation.

Let $m$ be the magnitude of $P$. Observe that for every gap $[i-1,i]$, $i=1,\ldots,m-1$, in 
  $\mathcal{C}$ there exist $x,y\in X$ such that
$r^*_x=i-1$ and $\ell^*_y=i$, that is $s_{\mathcal{C}}(x,y)=0$. Then by the assumption, $s_{\mathcal{R}}(x,y)=0$, it follows that  $r_x$ and $\ell_y$
are consecutive integers in ${\mathcal{R}}$, as well. Therefore, the interval endpoints in ${\mathcal{R}}$  form a sequence of $m$ consecutive integers.  

Because  the interval endpoints  in  the minimal endpoint representation ${\mathcal{R}}$ are $m$ consecutive integers,    ${\mathcal{R}}$ is a shift of  ${\mathcal{C}}$.
\end{proof}

\section{The full representation polyhedron}
\label{Fullsection}

\noindent
Given an interval order $P=(X,\prec)$, where $X=\{1,2,\ldots,n\}$, the {\it full representation polyhedron} $R_P\subset\mathbb{R}^{3n}$ is defined by  the real vectors $(\ell_1,r_1,\ldots,\ell_n,r_n,\rho_1,\ldots,\rho_n)$ satisfying for every $x,y\in X$ the inequalities:
\begin{eqnarray}\label{Rfull}
\left\{\qquad
\begin{array}{rcrl}
-\ell_y+r_x&\leq&-1& \text{if \ } x\prec y\\
\ell_x-r_y&\leq&0  & \text{if \ } x\| y, \  x\neq y\\
\ell_x-r_x+\rho_x &\leq& 0  \\
-\ell_x +r_x-\rho_x &\leq& 0 \\
\ell_x-r_x&\leq&0 \\
\end{array}
\right.
\end{eqnarray}

The system (\ref{Rfull}) can be written in form of $M x\leq b$, where the rows of $M$ correspond to the $0,\pm 1$ coefficients on the left-hand sides of the corresponding inequalities, its columns correspond to the $3n$ variables, $\ell_x, r_x$, and $\rho_x$, $x=1,\ldots,n$, and $b$ is the vector  determined by the right-hand side of the inequalities.  

\begin{proposition} 
\label{Rintegral} The full representation polyhedron 
$R_P$ defined by the system (\ref{Rfull})  is an integral polyhedron.
\end{proposition}
\begin{proof} 
We prove first by applying Theorem \ref{Ghouila-Houri} that the coefficient matrix $M$ of the system (\ref{Rfull}) is totally unimodular.
Let $A$ be an arbitrary submatrix of $M$ formed by selecting certain rows and columns.  Color all columns of $A$  black that correspond to a
 left endpoint or a right endpoint variable.  A column of $A$ corresponding to $\rho_x$ is also black if the column corresponding to $r_x$ 
 is not selected in $A$; otherwise, color this column white. 	
This is clearly an equitable bicoloring of the columns of $A$ for every submatrix $A$,
thus Theorem \ref{Ghouila-Houri} implies that $M$ is totally unimodular. 
Because the right hand side vector $b$ of (\ref{Rfull}) is integral, the polyhedron $R_P$ is integral, by Theorem \ref{HoffmanKruskal}.
\end{proof}

A linear inequality ${g}{x}\leq {d}$  is a {\it valid inequality} for $R_P$ if  $R_P\subseteq \{{ x}\in \mathbb{R}^{3n} : { g}{ x}\leq { d}\}$.
 If  ${ g}{ x}\leq { d}$ is a valid inequality for $R_P$, then 
   $F=\{{ x}\in R_P: { g}{ x}= { d}\}$ is called a {\it face} of $R_P$ defined by the inequality ${ g}{ x}\leq { d}$.  
   
 Let  $F$ be a face of $R_P$ defined by the inequality ${ g}{ x}\leq { d}$. If $F\neq\varnothing$ and 
  $F\neq R_P$, then  $H= \{{ x}\in \mathbb{R}^{3n} : { g}{ x}= { d}\}$ is called a {\it supporting hyperplane} to $R_P$; and $H\cap R_P$ is a face of $R_P$, by definition.
  
 A face $F$ is called a {\it facet} of $R_P$, provided $\dim(F)=\dim(R_P)-1$, (where  $\dim(S)=k$ if the maximum number of affinely independent  points in the set $S$ is equal to $k+1$). 
A face $F$ is called a {\it vertex}, provided $\dim(F)=0$.

\begin{theorem} [Canonical Line Theorem]
\label{main}
The intersection of the facets of  the full representation polyhedron $R_P$ is the line 
$\ell={ c} + t\cdot { k}$, where ${ c}$ is the $3n$-vector of  
the canonical representation of $P$, \  and ${ k}=(1,1,\ldots,1,1,0,\ldots,0)\in \mathbb{R}^{3n}$ is the characteristic vector of the first $2n$ location coordinates, and $t\in\mathbb{R}$.
\end{theorem}
\begin{proof}
Let $\mathcal{C}$  be the canonical representation of $P$, and let $S$ be the intersection of the facets of  $R_P$.

1. To see that $\ell\subseteq S$ we show that  $\ell\subset F$, for every facet  $F$ of $R_P$. Suppose that $F$ is
 defined by a valid inequality  ${ g}{ x}\leq { d}$ in (\ref{Rfull}), and   
let  $H= \{{ x}\in \mathbb{R}^{3n} : { g}{ x}= { d}\}$.  
By definition, ${ g}{ x}= { d}$, for every ${ x}\in F$.

By Proposition \ref{Rintegral}, $R_P$ is an integral polyhedron, therefore,  there is an integral point ${ f} \in F$ such that ${ g}{ f}= { d}$. 

Observe that the equality $g{ f}= { d}$ for the integral vector $f\in H$   means that  there is a slack zero pair in the  integral representation corresponding to $f$. By Theorem \ref{slackthm}, the same pair has slack zero in the canonical representation $\mathcal{C}$ as well; then it follows that ${ g}{ c}= { d}$ and implies  ${ c} \in F$. 

Because a slack value is invariant under shifting, $F=H\cap R_P$ contains the shifts of $c$ corresponding to the  integral representations ${ c} + t\cdot { k}$, that is,  ${ c} + t\cdot { k}\in F$, for every  integral $t$.
The integer polyhedron $R_P$ is the convex hull of its integral points, thus $\ell\subset F$ follows. 
 We conclude that $\ell\subseteq S$.
 
 2. Next we verify $S\subseteq \ell$.  Note that $S$ is an affine subspace, being the solution set of a system of linear equations.  
Consider an arbitrary pair $x,y\in X$ such that $y\not\prec x$ and $s_P(x,y)=0$.
Assume that $ { g}{ x}\leq { d}$ is the valid
inequality of (\ref{Rfull}) involving the pair $x,y\in X$, and let   $F= \{{ x}\in R_P : { g}{ x}= { d}\}$
be the corresponding face of $R_P$. 

Observe that $s_P(x,y)=0$ implies ${ c}\in  F$, that is $F\neq\varnothing$. On the other hand, 
$F\neq R_P$, because moving appropriately the interval endpoints in $\mathcal{C}$ a representation of $P$ is obtained easily with positive slack 
value between $x,y$, and as a result, the corresponding vector is in $R_P\setminus F$.
Thus the hyperplane  
 $H= \{{ x}\in \mathbb{R}^{3n} : { g}{ x}= { d}\}$ is a  {supporting hyperplane} to $R_P$, and
 $H\cap R_P=F$.  Since $ { c}\in S\cap H$, we obtain  $S\subset H$.
 
 Because $S$ is an affine subspace, and $R_P$ is an integral polyhedron, $S=\Conv(T)$, where $T$ is the set of  integral points in $S$.   
 Let ${ f}\in T$ be  arbitrary, and let $\mathcal{R}$ be the integral representation of $P$ corresponding to  ${ f}$. 
  Then we have  ${ f}\in S\subset H$
that implies  $s_\mathcal{R}(x,y)=0$. Therefore, 
 by the Converse Slack Theorem (Theorem \ref{conversethm}), $\mathcal{R}$ is a shift of $\mathcal{C}$. 
 We conclude that ${ f}\in \ell$ for every ${ f}\in T$, that is  $T\subset \ell$. Since $\ell$ is convex, $S=\Conv(T)\subseteq \ell$ follows.
\end{proof}

\section{The location polyhedron}
\label{Locsection}
The {representation polyhedron} $R_\epsilon^P\subset \mathbb{R}^{2n}$, with $\epsilon>0$,  of an interval order $P$ was introduced and  investigated  
by Doignon and Pauwels \cite{DoignonPauwels}. We use here the notation $D_P=R_1^P$, and call it the {\it location polyhedron} of $P$; the choice $\epsilon =1$ makes possible to use natural representations (in which the  interval  endpoints are non-negative  integers). 

 The system of inequalities defining  $D_P$ is obtained by substituting $\rho_x=r_x-\ell_x$ in the system (\ref{Rfull}), then adding the non-negativity inequalities $\rho_x\geq 0$ and
 $\ell_x\geq 0,\quad x=1,\ldots,n$. If $x\prec y$ then $\ell_x\leq r_x<\ell_y$, therefore, the inequality $\ell_y\geq 0$ is redundant in the system. After eliminating these redundant inequalities, we obtain the inequalities for every $x,y\in X$:
 \begin{eqnarray}\label{Dloc}
\left\{\qquad
\begin{array}{rcrl}
-\ell_x&\leq &0& \text{\ if $x$ is minimal in $P$} \\
-\ell_y+r_x&\leq&-1&  \text{\ if\ } x\prec y \\
\ell_x-r_y&\leq&0&   \text{\ if\ } x\| y, x\neq y\\
\ell_x-r_x&\leq&0&
\end{array}
\right.
\end{eqnarray}
Note that $D_P$ is obtained by projecting $R_P$ into the subspace of the first $2n$ location coordinates, 
then cutting  off pieces not in the positive quadrant. Let 
 ${c}_{loc}\in\mathbb{R}^{2n}$  be the location vector of the canonical representation of $P$. 
 Doignon and Pauwels \cite[Theorem 5]{DoignonPauwels} determined the facet-defining inequalities of  $D_P$, and derived that $D_P$ is a pointed affine cone  with its apex at  ${c}_{loc}$. 
A short direct proof is repeated here not relying on the explicit  list of facet-defining inequalities.

\begin{theorem}{\em (Doignon and Pauwels)} 
\label{DPcone}
 The location polyhedron $D_P$  is
a pointed affine cone  with its apex at  the canonical representation of $P$.
\end{theorem}
\begin{proof} We use the observation  \cite[Proposition 5]{DoignonPauwels} that  $D_P$ is full-dimensional; and as a corollary, (\ref{Dloc}) contains a unique minimal subsystem of facet-defining inequalities. 

By Theorem \ref{main}, the  canonical representation ${c}_{loc}=(\ell^*_1,r^*_1,\ldots, \ell^*_n,r^*_n)$ satisfies the facet-defining inequalities in (\ref{Rfull}), and hence the same inequalities in (\ref{Dloc}). Therefore, the point ${c}_{loc}$ is contained by all facets of $D_P$ as well. 

We verify next that  ${c}_{loc}$ is a vertex of $D_P$; equivalently, 
 there are no points ${x_1}, {x_2}\in$$ D_P\setminus\{{c}_{loc}\}$ and
$\lambda\in\mathbb{R}$, $0<\lambda<1$, such that ${c}_{loc}=$ $\lambda{x_1}+(1-\lambda){x_2}$.

For $i=1,2$ let ${x_i}=$$(\ell^{(i)}_1,r^{(i)}_1,\ldots, \ell^{(i)}_n,r^{(i)}_n)$. 
Since ${ x_1}\neq $ ${{c}}$$_{loc}$, there are coordinates where the vectors  disagree. Suppose w.l.o.g. that $\ell^{(1)}_1\neq \ell^*_1$ (the case $r^{(1)}_1\neq  r^*_1$ can be resolved similarly).
Then by Proposition \ref{minCanon}, we have $\ell^{(1)}_1> \ell^*_1$, and $\ell^{(2)}_1\geq \ell^*_1$.
Thus we obtain
\begin{equation*}
\ell^*_1=\lambda\ell^{(1)}_1+(1-\lambda)\ell^{(2)}_1>\lambda \ell^*_1+(1-\lambda) \ell^*_1= \ell^*_1,
\end{equation*}
a contradiction. Because ${{c}}$$_{loc}$ is a vertex contained by every facet of $D_P$, the theorem follows.
\end{proof}

 Doignon and Pauwels \cite[Theorem 4]{DoignonPauwels}  determined the facet-defining inequalities of the location polyhedron $D_P$. Here we present a short proof of their result, which emphasizes  our 
approaches throughout the present work that the canonical representation is a powerful model that encompasses full information of an interval order
and adds a combinatorial geometry view when investigating its structure.

\begin{theorem}{\em (Doignon and Pauwels)}  
\label{DPfacet} Let ${c}_{loc}=(\ell^*_1,r^*_1,\ldots, \ell^*_n,r^*_n)$  be the canonical representation of $P$. 
The facet-defining inequalities of the location polyhedron $D_P$  are for every $x,y\in X$:
 \begin{eqnarray}\label{Dirred}
\left\{\qquad
\begin{array}{rcrl}
-\ell_x&\leq &0,  & \text{if \ $x$ is minimal in $P$},\\
-\ell_y+r_x&\leq&-1, &  \text{if \  }\ell^*_y=r_x^*+1,\\
\ell_x-r_y&\leq&0,  & \text{if \ }  \ell^*_x=r_y^*, \ x\neq y,\\ 
\ell_x-r_x&\leq&0,  &   \text{if \ }  \ell^*_x=r_x^*.
\end{array}
\right.
\end{eqnarray}
\end{theorem}
\begin{proof}
The Slack Theorem (Theorem \ref{slackthm}) implies that (\ref{Dirred}) as the subsystem of (\ref{Dloc}) defines $D_P$. To see that (\ref{Dirred}) is an irredundant system of inequalities, a straightforward case by case test shows that dropping 
any inequality from  the system (\ref{Dirred}) results in a subsystem with a feasible solution set larger than $D_P$.

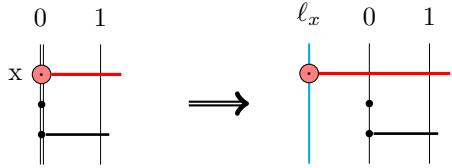
\begin{figure}[htp]
\begin{center}
\begin{tikzpicture}[scale=.8] 
\foreach \i in {0,1}{
   \draw[line width=.01em](\i,0)--(\i,2);}
    \draw[line width=.01em](0.05,0)--(0.05,2);
\node[label=above:{0}]()at(0,2){};
\node[label=above:{1}]()at(1,2){};

\node[A](L)at(0.02,.5){};\node[](R)at(1.3,.5){};\draw[line width=.1em](L)--(R);
\node[A](L)at(0.02,1){};
\node[Xr,label=left:{x}](L)at(0.02,1.5){.};\node[](R)at(1.5,1.5){};\draw[line width=.12em,red,](L)--(R);
\end{tikzpicture}
\hskip.5cm
\begin{tikzpicture}[scale=.8]
\draw[double, line width=.08em,->](-3,1)--(-2,1);
\foreach \i in {0,1}{
   \draw[line width=.01em](\i,0)--(\i,2);
 \node[label=above:{\i}]()at(\i,2){};  }
  \draw[line width=.08em,cyan](-1,0)--(-1,2);
 \node[label=above:{$\ell_x$}]()at(-1,2){}; 
\node[A](L)at(0,.5){};\node[](R)at(1.3,.5){};\draw[line width=.1em](L)--(R);
\node[A](L)at(0,1){};
\node[Xr](L)at(-1,1.5){.};\node[](R)at(1.5,1.5){};\draw[line width=.12em,red,](L)--(R);
\end{tikzpicture}
\end{center}
\caption{\quad  $x$ is minimal in $P$ \ and \ $-\ell_x> 0$} 
\label{1}
\end{figure}

The inequalities of (\ref{Dirred}) are interpreted on the diagram of the canonical representation of $P$. 
The violation of just one inequality is visualized by a local surgery on this diagram. 
 In this way we obtain a proof (almost without words) using Figures \ref{1} - \ref{4},  where one of the   three types of inequalities in the system (\ref{Dirred}) is violated, all the other ones remain satisfied.

The simplest surgery in  Fig.\ref{1} shows a feasible solution not in $D_P$ due to the introduced negative coordinate. 
In all other surgery a pair of endpoints is involved corresponding to an inequality selected by a slack zero pair of $P$.
 
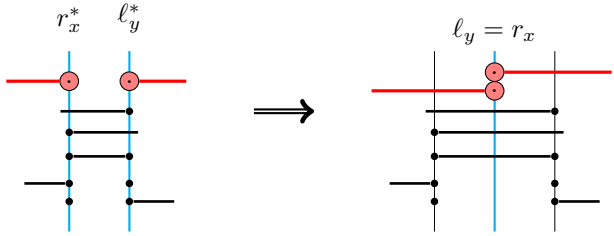
\begin{figure}[htp] 
\begin{center}
\begin{tikzpicture}[scale=.8]
  \foreach \i in {-1,...,0}{
   \draw[line width=.08em,cyan](\i,-1)--(\i,2);}
     
\node[label=above:{$\ell_y^*$}]()at(0,2){};
\node[label=above:{$r_x^*$}]()at(-1,2){};

\node[A](R)at(0,.25){};\node[A](L)at(-1,.25){};\draw[line width=.1em](L)--(R);
\node[A](L)at(-1,.65){};\node[](R)at(0.3,.65){};\draw[line width=.1em](L)--(R);
\node[A](L)at(0,1){};\node[](R)at(-1.3,1){};\draw[line width=.1em](L)--(R);
\node[Xr](L)at(0,1.5){.}; \node[](R)at(1.1,1.5){};\draw[line width=.12em,red,](L)--(R);
\node[Xr](L)at(-1,1.5){.}; \node[](R)at(-2.2,1.5){};\draw[line width=.12em,red,](L)--(R);

\node[A](L)at(0,-.5){};\node[](R)at(.9,-.5){};\draw[line width=.1em](L)--(R);
\node[A](L)at(-1,-.2){};\node[](R)at(-1.9,-.2){};\draw[line width=.1em](L)--(R);
\node[A](L)at(-1,-.5){}; \node[A](L)at(0,-.2){};
\end{tikzpicture}
\hskip.5cm
\begin{tikzpicture}[scale=.8]
\draw[double, line width=.08em,->](-4,1)--(-3,1);
  \foreach \i in {-1,1}{
   \draw[line width=.01em](\i,-1)--(\i,2);}
    \draw[line width=.08em,cyan](0,-1)--(0,2);
       
\node[label=above:{$\ell_y=r_x$}]()at(0,1.8){};

\node[A](R)at(1,.25){};\node[A](L)at(-1,.25){};\draw[line width=.1em](L)--(R);
\node[A](L)at(-1,.65){};\node[](R)at(1.3,.65){};\draw[line width=.1em](L)--(R);
\node[A](L)at(1,1){};\node[](R)at(-1.3,1){};\draw[line width=.1em](L)--(R);
\node[Xr](L)at(0,1.65){.}; \node[](R)at(2.1,1.65){};\draw[line width=.12em,red,](L)--(R);
\node[Xr](L)at(0,1.34){.}; \node[](R)at(-2.2,1.34){};\draw[line width=.12em,red,](L)--(R);

\node[A](L)at(1,-.5){};\node[](R)at(1.9,-.5){};\draw[line width=.1em](L)--(R);
\node[A](L)at(-1,-.2){};\node[](R)at(-1.9,-.2){};\draw[line width=.1em](L)--(R);
\node[A](L)at(-1,-.5){}; \node[A](L)at(1,-.2){};
\end{tikzpicture}

\end{center}
\caption{\quad   $\ell^*_y=r_x^*+1$ \ and \  $-\ell_y+r_x > -1$ } 
\label{3}
\end{figure}
  
\begin{figure}[htp] 
\begin{center}

\begin{tikzpicture}[scale=.8]

  \foreach \i in {-1,1}{
   \draw[line width=.01em](\i,-.5)--(\i,2);}
    \draw[line width=.08em,cyan](0,-.5)--(0,2);
       
\node[label=above:{$\ell_x^*=r_y^*$}]()at(0,1.8){};

\node[A](R)at(0,.25){};
\node[A](L)at(0,.65){};\node[](R)at(1.3,.65){};\draw[line width=.1em](L)--(R);
\node[A](L)at(0,1){};\node[](R)at(-1.3,1){};\draw[line width=.1em](L)--(R);
\node[Xr](L)at(0,1.65){.}; \node[](R)at(2.1,1.65){};\draw[line width=.12em,red,](L)--(R);
\node[Xr](L)at(0,1.34){.}; \node[](R)at(-2.2,1.34){};\draw[line width=.12em,red,](L)--(R);

\node[A](L)at(1,-.2){};\node[](R)at(1.9,-.2){};\draw[line width=.1em](L)--(R);
\node[A](L)at(-1,-.2){};\node[](R)at(-1.9,-.2){};\draw[line width=.1em](L)--(R);
\end{tikzpicture}
\begin{tikzpicture}[scale=.8]
\draw[double, line width=.08em,->](-4,1)--(-3,1);
  \foreach \i in {-1,...,0}{
   \draw[line width=.08em,cyan](\i,-.5)--(\i,2);}
       \foreach \i in {-2,1}{
   \draw[line width=.01em](\i,-.5)--(\i,2);}
   
\node[label=above:{$\ell_x$}]()at(0,1.8){};
\node[label=above:{$r_y$}]()at(-1,1.8){};

\node[A](R)at(0,.25){};\node[A](L)at(-1,.25){};\draw[line width=.1em](L)--(R);
\node[A](L)at(-1,.65){};\node[](R)at(1.3,.65){};\draw[line width=.1em](L)--(R);
\node[A](L)at(0,1){};\node[](R)at(-2.3,1){};\draw[line width=.1em](L)--(R);
\node[Xr](L)at(0,1.5){.}; \node[](R)at(1.8,1.5){};\draw[line width=.12em,red,](L)--(R);
\node[Xr](L)at(-1,1.5){.}; \node[](R)at(-2.7,1.5){};\draw[line width=.12em,red,](L)--(R);

\node[A](L)at(1,-.2){};\node[](R)at(1.9,-.2){};\draw[line width=.1em](L)--(R);
\node[A](L)at(-2,-.2){};\node[](R)at(-2.7,-.2){};\draw[line width=.1em](L)--(R);
\end{tikzpicture}

\end{center}
\caption{\quad    $\ell^*_x=r_y^* $ \ and \ $\ell_x-r_y> 0$,} 
\label{4}
\end{figure}
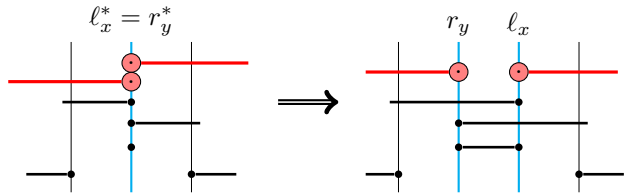
 When  removing and adding new endpoints, each surgery includes necessary adjustments of  the endpoints to maintain integer values, while not changing the order between pairs of $P$, except the pair involved in the violated inequality.
 If $x\neq y$ then the feasible solutions of the modified systems in Figures \ref{3} and \ref{4} are interval representations of a poset different from $P$. If $x=y$ then the surgery results in a feasible vector with $r_x<\ell_x$ that is not even an interval representation, since the right endpoint of an interval  cannot lie below its left endpoint.
\end{proof}
Pairs $(x,y)$ satisfying  $\ell^*_y=r_x^* +1$ or $\ell^*_x=r_y^* $, $y\neq x$, in  the linear system (\ref{Dirred}) are  called by Doignon, Pauwels  \cite[Proposition 3]{DoignonPauwels} noses or hollows, respectively. In the last inequality $x\in X$ is called contractible when $\ell^*_x=r_x^*$.
These pairs in an interval order are related to standard concepts  of cover and critical pairs in poset theory  (cf. Trotter \cite{TrotterBook}); they also appear in our investigations as slack zero pairs:  
 \begin{itemize}
\item $(x,x)$ is a {\it slack zero pair}  if
$\ell^*_x=r^*_x$; 
\item $(x,y)$ is a {\it slack zero cover pair} if  $\ell^*_y=r_x^* +1$; 
\item  $(x,y)$ is called a {\it sharp pair} if $\ell^*_x=r_y^*$, $y\neq x$, 
equivalently, if $(y,x)$ is a  slack zero critical pair. 
\end{itemize}
Then the system (\ref {Dirred}) defining $D_P$ becomes: for every $x,y\in X$, 
 \begin{eqnarray}\label{defDP}
\left\{\qquad
\begin{array}{rcrl}
-\ell_x&\leq &0  & \text{\ if $x$ is a minimal element}\\
-\ell_y+r_x&\leq&-1 &  \text{\ if \  $(x,y)$ is a slack zero cover pair}\\
\ell_x-r_y&\leq&0  & \text{\ if \ $(x,y)$ is a sharp pair}\\
\ell_x-r_x&\leq&0  &  \text{\ if \ $(x,x)$ is  a slack zero pair}.\\ 
\end{array}
\right.
\end{eqnarray}

\section{The length polyhedron}
\label{Lengthsection}

The length polyhedron  $Q_P$ of an interval order $P=(X,\prec)$ is obtained from  the {full representation polyhedron} $R_P\subset \mathbb{R}^{3n}$  by a projection into $\mathbb{R}^{n}$, that is, by keeping the last $n$ length coordinates of each vector of $R_P$. Similar to the location polyhedron $D_P$ (see Theorem \ref{DPcone}) the length polyhedron $Q_P$ is a conic polyhedron.

\begin{theorem}
\label{Qcone}
The length polyhedron $Q_P$ of an interval order $P$ is an affine cone with its apex at the length vector of the canonical representation of $P$.
\end{theorem}
\begin{proof}
Due to  the Canonical Line Theorem (Theorem \ref{main}), the
projection of the 
 intersection of the  facets of the full polyhedron $R_P$ 
 is equal to the vector of interval lengths in the canonical representation of $P$. Then the minimality of the
 canonical representation expressed in Proposition \ref{minCanon} implies  that the length  vector of the canonical representation is the apex of $Q_P$.
 \end{proof}

The linear system defining  $Q_P$ can be obtained from the minimal system (\ref{defDP}) that was 
derived for  the location polyhedron $D_P$ in Theorem \ref{DPfacet}.
The transformation proceeds in two steps, 
first we discard the variables $r_x$ from (\ref{defDP}) by substituting $r_x=\ell_x+\rho_x$ for every $x\in X$. 
This leads to an equivalent linear system containing the inequalities
 for every $x,y\in X$:
\begin{eqnarray}\label{LRO}
\left\{
\begin{array}{rcll}
 \ell_x+\rho_x+1&\leq& \ell_y\  & \text{\ if $(x,y)$ is a slack zero cover pair},\ \\
\ell_x&\leq& \ell_y+\rho_y\  & \text{\ if $(x,y)$ is a sharp pair},\\
-\rho_x&\leq& 0 &  \text{\ if \ $(x,x)$ is a slack zero  pair}.
\end{array}
\right.
\end{eqnarray}
Note that the  inequalities of the form $-\ell_x\leq 0$  in (\ref{defDP}) are meaningless for $Q_P$.
In the second step of the transformation we discard the variables 
 $\ell_x$ from (\ref{LRO})  by using Fourier-Motzkin elimination (FME). For this step we introduce a directed graph theory model.

\subsection*{The \FaceGraph  of $P$}
 To help characterize the length polyhedron $Q_P$ of an interval order $P=(X,\prec)$ we build a blue/red colored directed graph with arcs weighted with the inequalities of the system (\ref{LRO}).  
The {\it \FaceGraph} $G_P$ 
has vertex set $V=\{\rho_1,\rho_2,\ldots,\rho_n\}$; and   
 it has  red loops, and blue or red arcs between distinct  $\rho_x,\rho_y\in V$  with arc weights as follows:
\begin{eqnarray}\label{FaceGraphWeights}
\left\{	\begin{array}{ccll}
	Arc &Weight\\  
 \rho_x\red \rho_x \ &- \rho_x \leq 0 & \Leftrightarrow &  \text{$(x,x)$ is a slack zero  pair},\\ 
    \rho_x\blue \rho_y \  &\ell_x + \rho_x + 1 \leq \ell_y  & \Leftrightarrow & (x,y) \text{ is a slack zero cover pair}, \\ 
     \rho_x\red \rho_y\   &\ell_x \leq \ell_y + \rho_y & \Leftrightarrow & (x,y) \text{ is a sharp pair}.
\end{array}\right.
\end{eqnarray}  

Fig.\ref{FMexample} depicts the canonical representation of  the interval order $P=[0,1,2,1,0,2,3]$ and its \FaceGraph

 \def\XAxisBump{.7} 
\def\MaxVerticalLine{4}  
\def\MinVerticalLine{0}  
\def\IntervalThickness{1.4}
\def\Gap{1.8} 
\def\VGap{.5}   
\def\Epsilon{0.3}  
\def\Height{3} 

\begin{figure}[H]
	\begin{center}
		\begin{minipage}{.1\textwidth}
		\end{minipage}\hskip1.15cm
		\begin{minipage}{.45\textwidth}
	\begin{tikzpicture}[scale=.8]
	\node () at (7,1){};
\foreach \i in {0,...,4}{
    \draw[line width=.02em](3+\i,4)--(3+\i,6);
 \node[]()at(3+\i,6.2){\tiny{\i}};   }

\node[A,label=left:1](L) at (3,5) {};
\node[A,label=left:2]() at (4,4.5) {};

\node[A,label=left:6](L) at (5,4.5) {};
\node[A](R) at (6,4.5) {};
\draw[line width=.1em](L)--(R); 

\node[A,label=left:5](L) at (3,5.5) {};
\node[A](R) at (5,5.5) {};
\draw[line width=.1em](L)--(R); 

\node[A,label=left:7](L) at (6,5) {};
\node[A](R) at (7,5) {};
\draw[line width=.1em](L)--(R);

\node[A,label=left:4](L)at (4,5) {};
\node[A](R) at (5,5) {};
\draw[line width=.1em](L)--(R); 

\node[A,label=left:3]()at (7,4.5) {};
\end{tikzpicture}
		\end{minipage}
\begin{tikzpicture}[scale=.75]

\node[A,label=left:{ $\rho_4$}](2) at (1,0) {};
\node[A,label=left:{ $\rho_1$}](1) at (3,-1) {}; 

\node[A,label=right:{ $\rho_5$}](5) at (5,0) {};

\node[A,label=left:{ $\rho_6$}](6) at (0,2) {};
\node[A,label=right:{ $\rho_7$}](4) at (5,3) {};

\node[A,label=left:{ $\rho_3$}](3) at (0,4) {};

\node[A,label=right:{ $\rho_2$}](7) at (3,5) {};
\draw[line width=.1em,red,->>](2)--(7);
\draw[line width=.1em,cyan,->](1)--(7);
\draw[line width=.1em,cyan,->>](7)--(6);
\draw[red,line width=.1em,->] (7) edge[out=45, in=135, looseness=12] (7);

\draw[line width=.1em,red,->>](3)--(4);
\draw[line width=.1em,red,->>](4)--(6);
\draw[line width=.1em,red,->](5)--(1);
\draw[line width=.1em,cyan,->](5)--(4);
\draw[line width=.1em,cyan,->>](2)--(4);
\draw[line width=.1em,red,->](6)--(5);
\draw[line width=.1em,red,<-](2)--(6);
\draw[line width=.1em,cyan,->](1)--(2);

\draw[line width=.1em,cyan,->](6)--(3);
\draw[red,line width=.1em,->] (1) edge[out=200, in=-65, looseness=12] (1);
\draw[red,line width=.1em,->] (3) edge[out=45, in=135, looseness=12] (3);

\end{tikzpicture}
	\end{center}	
\vskip-2cm	
	\caption{Canonical representation of  $P=[0,1,2,1,0,2,3]$ and its \FaceGraph}
\label{FMexample}
\end{figure}
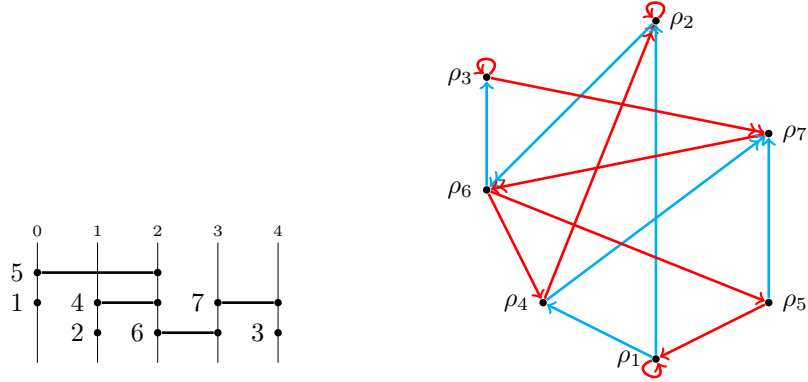

Observe that the sum of the two arc weights belonging to a 3-path $\rho_x\black\rho_y\black\rho_z$
is an inequality that does not contain $\ell_y$. 

Therefore, the sum of arc weights along a directed cycle of $G$ yields an inequality  with $\rho$-variables only. We show next that 
 every such {\it cycle inequality}  has the form 
\begin{eqnarray}\label{CYC}
 \gamma \ +\sum_{i\in A}\rho_i \ \leq \ \sum_{j\in B}\rho_j\  ,
 \end{eqnarray}
where $\gamma$ is a non-negative integer and $A,B\subset X$ are disjoint, and they are determined as follows. 

\begin{proposition}
\label{CYCLE INEQUALITY RULE}
Let $C\subset G$ be a directed cycle of length at least two of the \FaceGraph. For every  $3$-path $T=(\rho_x\rightarrow \rho_s\rightarrow \rho_y)$  along  $C$,  $s\in A$ if both arcs of $T$ are blue, $s\in B$ if both arcs of $T$ are red, and $s\notin A\cup B$, otherwise. The constant $\gamma$ counts the number of  blue arcs along $C$.
A loop is considered as a singleton red cycle $C=(\rho_s)$,  and it determines the cycle inequality $0\leq \rho_s$.
\end{proposition}
\begin{proof} Let $W$ be the inequality obtained as the sum of the arc weights in $C$. According to the color pattern of $T$ we obtain:

\begin{itemize}
\item for $T=(\rho_x\red \rho_s\red \rho_y)$  the weight sum is \ $\ell_x\leq \ell_y+\rho_s+\rho_y$, hence $s\in B$ 
\item for  $T=(\rho_x\blue \rho_s\blue \rho_y)$ the weight sum is \  $\ell_x+\rho_x+\rho_s+2\leq \ell_y$,  hence $s\in A$;
\item for $T=(\rho_x\blue \rho_s\red \rho_y)$ the label sum is  \ $\ell_x+\rho_x+1\leq \ell_y+\rho_y$, hence $s\not\in A\cup B$
\item for $T=(\rho_x\red \rho_s\blue \rho_y)$ the label sum is \ $\ell_x+1\leq \ell_y$, hence $s\not\in A\cup B$.  
\end{itemize}
Notice that each blue arc adds $1$ to the left of $W$, thus $\gamma$ counts the blue arcs of $C$. 
Therefore, $W$ takes the form in (\ref{CYC}) as stated.
\end{proof}

\begin{proposition} \label{CycleInequalitiesDefineLengthPolyhedron} 
The vector  $(\ell_1,\ldots,\ell_n, \rho_1,\ldots,\rho_n)$ is a feasible solution of the  linear system (\ref{LRO}) if and only if \ $\ell_i\geq 0$, for $1\leq i\leq n$, and the length vector $( \rho_1,\ldots,\rho_n)$ 
is a feasible solution of the system (\ref{CYC}) of all cycle inequalities. 
\end{proposition}
\begin{proof}
Every cycle inequality  in (\ref{CYC}) is obtained as the sum of the arc weight inequalities in  (\ref{LRO}), thus necessity follows.   For the converse,  
we verify that every cycle inequality of the system (\ref{CYC})
is obtained by 
 Fourier-Motzkin elimination (FME) that cancels the left endpoint variables $\ell_i$, $i=1\ldots,n$, from the
 linear system (\ref{LRO}). 

 Note that a given variable $\ell_s$, $1\leq s\leq n$, occurs in weights assigned to arcs of $G_P$ incident with vertex $\rho_s$. Moreover, 
 the weight of an arc $\rho_s\rightarrow \rho_y$ of any color is an upper bound for $\ell_s$, 
 and  
  the weight of an arc $\rho_x\rightarrow \rho_s$ of any color is  a lower bound  for $\ell_s$. 
 
The linear system  (\ref{LRO}) is going through stages $s=1,\ldots,n$, as FME progresses. 
In  Stage $s$, all lower bound inequalities and all upper bound inequalities for variable $\ell_s$ are combined to cancel out $\ell_s$. Stage $s$ starts  with a system equivalent to (\ref{LRO}), where no variable $\ell_i$ is present for $i<s$, and its  task  consists of eliminating the next variable $\ell_{s}$. 
FME terminates at stage  $s= n$, until  
all variables $\ell_1,\ldots,\ell_n$ are eliminated from the linear system. 

The elimination of $\ell_s$ in Stage $s$, $s=1,\ldots,n$, is realized in several steps. 
 To keep a record of the changes of the linear system FME is making during the stages  
  we introduce a sequence of auxiliary weighted digraphs $G_{FM}^{(s)}$, $s=1,\ldots,n$, called {\it stage graphs}.
   The vertex set of every stage graph
    is $V_{FM}=\{{1},\ldots,n\}$, where  $j\in V_{FM}$  indicates the  
 correspondence of $\rho_j$ and $\ell_j$. 
The stage graphs record the consecutive steps made by FME  by dynamically changing within each stage. 
There are two types of steps to make in a given Stage $s$:\\
\begin{itemize}
  \item[(a)] processing a 3-path (x\black s\black y) of $ G_{FM}^{(s)}$  for $s+1\leq x,y\leq n$, $x\neq y$,
   \item[(b)] processing a 2-cycle (x\black s\black x) in $G_{FM}^{(s)}$,  for  $s+1\leq x\leq n$. \\  
 \end{itemize}
 
In a step of type (a) for a 3-path (x\black s\black y),  $x\neq y$, of  the stage graph  a new arc x \new y is introduced in to $G_{FM}^{(s)}$. The weight of x \new y  is an inequality free from $\ell_s$ obtained as the sum
 of  the weights of  x\black s  and  s\black y. 
 
In this step a directed cycle of the stage graph going through vertex $s$ is retracted by cutting off $s$ and adding the new arc. A step of type (a)  is called a {\it retracting step}. Each arc, original arcs and new ones, is assigned its {\it retraction history}, a subpath of the \FaceGraph. The retraction history of an original arc x\black y in $G_P$ is the 2-path $(x,y)$. The {retraction history} of a new arc x \new y in Stage $s$ is the concatenation of the subpaths  assigned as the retraction history of  the arc x \new s and s \new y.  
  
In a step of type (b), called an {\it inflating step}, the cycle of $G_P$ is recovered that was retracted in some earlier retracting steps of type (a). When the 2-cycle x \new s \new x is processed  in an inflating step of Stage $s$,  the retracting history of the arcs
 x \new s  and   s \new x is a directed xs-path and a directed sx-path, respectively. The concatenation of these disjoint directed paths of $G_P$ yields a cycle of the \FaceGraph and is inserted to the cycle bin $\Gamma$.  

Stage $s$ starts with the inspection of all arcs going into vertex $s$  and going out from $s$ in the initial stage graph 
that is equal to the final  state of $G_{FM}^{(s-1)}$.
Note that all steps in Stage $s$ can be performed in arbitrary order. 
After no more steps remain to process, all arcs incident with vertex $s$ are removed, thus
making $s$ an isolated vertex, 
the resulting graph become the final state of $G_{FM}^{(s)}$. The whole procedure terminates when $G_{FM}^{(s)}$ consists of isolated vertices.
  
The FME is initialized with  $G_{FM}^{(0)}\cong G_P$ forming the input to Stage 1.
At the end of Stage $s$  all vertices $1,\ldots, s$ become isolated
 in the stage graph $G_{FM}^{(s)}$. These isolated vertices indicate that the corresponding variables $\ell_1,\ldots,\ell_s$  are eliminated from the linear system. 
 
 Each inequality free from left endpoint variables is obtained as the weight sum of a directed cycle of the \FaceGraph.
 Therefore, finding the linear system where the left endpoint variables are all eliminated is reduced to finding all directed cycles.

CLAIM 1. {\it  The weight of an arc x \black y in $G_{FM}^{(s)}$ is an inequality of the form
\begin{equation}\label{arclabel}
\ell_x +  f + \alpha \leq h+ \ell_y,
\end{equation}
 where $f,h$ are  expressions containing only $\rho$-variables, and $\alpha$ is a non-negative integer.}

\begin{proof}
For $x\neq y$, the inequality of the form (\ref{arclabel}) is an  upper bound for $\ell_x$ in terms of $\ell_y$, and it is a lower bound for $\ell_y$ in terms of $\ell_x$. This is true  for the weight of a red or a blue colored arc, x \red y or x \blue y,  in the initial stage graph $G_{FM}^{(0)}$ (that is isomorphic to the \FaceGraph $G_P$); indeed these inequality weights  are $\ell_x \leq \rho_y + \ell_y$ and $\ell_x +\rho_x+1\leq \ell_y$,
respectively.  

Assume that  x \new y  ($x\neq y$) is a new compound edge added to the stage graph in Stage s, and assume by induction that the weights of x \new s and s \new y are 
 $\ell_x +f_1+\alpha_1 \leq h_1 +\ell_s$ and $\ell_s +f_2+\alpha_2 \leq h_2 +\ell_y$, respectively. Thus 
  $\ell_x +(f_1+f_2)+(\alpha_1+\alpha_2) \leq (h_1 + h_2)+\ell_y$ 
 becomes the weight of x \new y; it has the form  (\ref{arclabel}) thus concluding the proof of Claim 1.
\end{proof}
 
 Due to Claim 1  a left end point variable $\ell_s$ occurs in the weight inequality of an arc if and only if the arc is incident with vertex $s$, therefore, steps (a) and (b) in Stage s of the FME procedure eliminate $\ell_s$ from the linear system. 
 
The formal FME procedure maintains an extending  list  $\Gamma= (\Gamma_1,\Gamma_2,\ldots)$ of the cycles  of the  \FaceGraph. This list is initialized with all 1-vertex cycles, that is with the red loops. \\

CLAIM 2. {\it The FME procedure explores all cycle inequalities.}\\

Before the straightforward induction proof of Claim 2 we give 
an example that illustrates the stages and the steps that the FME procedure is making. Fig.\ref{FMexample}  displays  the canonical representation and the \FaceGraph $G_P$ of  the interval order $P=[0,1,2,1,0,2,3]$.

 It takes 6 stages to explore all cycles of $G_P$. The description of each Stage $s$, $1\leq s\leq 6$, starts with the list of the arcs incident with vertex $s$ and their retraction history. This part is followed by the list of the steps of type (a) and (b), the new arcs and their retraction history. Next to this list is a diagram picturing the initial state of the  stage graph $G_{FM}^{(s)}$. Dashed arcs indicate all those arcs that will be removed in the final state of $G_{FM}^{(s)}$ making $s$ isolated.

The list $\Gamma$ where the cycles are collected  is initialized with three 1-cycles, 
$\Gamma_1, \Gamma_2$ and  $\Gamma_3$,  corresponding to the red loops at $\rho_1,\rho_2$ and $\rho_3$ in the \FaceGraph.
The six FME stages are presented in the tables and diagrams below. The left column of each table lists the arcs, the right column contains the corresponding retraction history of the arcs.

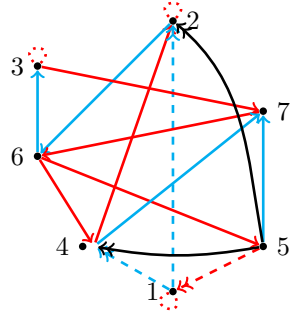
\begin{minipage}{0.49\textwidth}
\begin{tabular}{l|lc}
&STAGE 1\\
\hline\hline
5\black 1 & (5,1)\\
1\black 2 &(1,2)\\
1\black 4 & (1,4)\\
\hline
5\black 2 &(5,1,2)&\\
5\black 4 &(5,1,4)&\\
\hline\hline
\end{tabular}
\end{minipage}
\begin{minipage}{0.45\textwidth}
\begin{tikzpicture}[scale=.6]
	
\node[A,label=left:{ $4$}](4) at (1,0) {};
\node[A,label=left:{ $1$}](1) at (3,-1) {}; 

\node[A,label=right:{ $5$}](5) at (5,0) {};

\node[A,label=left:{ $6$}](6) at (0,2) {};
\node[A,label=right:{ $7$}](4) at (5,3) {};

\node[A,label=left:{ $3$}](3) at (0,4) {};
\node[A,label=right:{ $2$}](7) at (3,5) {};
\draw[line width=.1em,red,->](2)--(7);
\draw[line width=.1em,cyan,->>,dashed](1)--(7);
\draw[line width=.1em,cyan,->](7)--(6);
\draw[red,line width=.1em,dotted] (7) edge[out=45, in=135, looseness=12] (7);

\draw[line width=.1em,red,->](3)--(4);
\draw[line width=.1em,red,->](4)--(6);
\draw[line width=.1em,red,->>,dashed](5)--(1);
\draw[line width=.1em,cyan,->>,dashed](1)--(2);
\draw[line width=.1em,cyan,->](5)--(4);
\draw[line width=.1em,red,->](6)--(5);
\draw[line width=.1em,red,<-](2)--(6);
\draw[line width=.1em,cyan,->](2)--(4);
\draw[line width=.1em,->>] (5) edge[out=-170, in=-10] (2);
\draw[line width=.1em,->>] (5) edge[out=100, in=-40] (7);
\draw[line width=.1em,cyan,->](6)--(3);
\draw[red,line width=.1em,dotted] (1) edge[out=200, in=-65, looseness=12] (1);
\draw[red,line width=.1em,dotted] (3) edge[out=45, in=135, looseness=12] (3);
\end{tikzpicture}
\end{minipage}\\
\vskip.25cm

\begin{minipage}{0.49\textwidth}
\begin{tabular}{l|lc}
&STAGE 2\\
\hline\hline
4\black 2 & (4,2)\\
5\black 2 &(5,1,2)\\
2\black 6 & (2,6)\\
\hline
5\black 6&(5,1,2,6)&\\
4\black 6 &(4,2,6)&\\
\hline\hline
\end{tabular}
\end{minipage}
\begin{minipage}{0.45\textwidth}
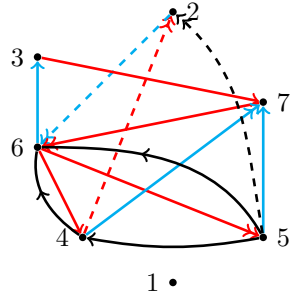
\begin{tikzpicture}[scale=.6]

\node[A,label=left:{ $4$}](2) at (1,0) {};
\node[A,label=left:{ $1$}](1) at (3,-1) {}; 

\node[A,label=right:{ $5$}](5) at (5,0) {};

\node[A,label=left:{ $6$}](6) at (0,2) {};
\node[A,label=right:{ $7$}](4) at (5,3) {};

\node[A,label=left:{ $3$}](3) at (0,4) {};
\node[A,label=right:{ $2$}](7) at (3,5) {};
\draw[line width=.1em,red,->>,dashed](2)--(7);

\draw[line width=.1em,cyan,->>,dashed](7)--(6);

\draw[line width=.1em,red,->](3)--(4);
\draw[line width=.1em,red,->>](4)--(6);

\draw[line width=.1em,cyan,->](5)--(4);
\draw[line width=.1em,red,->>](6)--(5);
\draw[line width=.1em,red,<-](2)--(6);
\draw[line width=.1em,cyan,->>](2)--(4);

\draw[line width=.1em,cyan,->](6)--(3);

\draw[,line width=.1em,->>,dashed] (5) edge[out=100, in=-40] (7);
\draw[,line width=.1em,->] (5) edge[out=-170, in=-10] (2);

\draw[line width=.1em,
postaction={decorate, decoration={markings, mark=at position 0.6 with {\arrow{>}}}}]  
(2) to[out=145, in=-100] (6);
\draw[line width=.1em,
postaction={decorate, decoration={markings, mark=at position 0.6 with {\arrow{>}}}}]  
(5) to[out=120, in=0] (6);
\end{tikzpicture}
\end{minipage}
\newpage

\begin{minipage}{0.49\textwidth}
\begin{tabular}{l|lc}
&STAGE 3\\
\hline\hline
6\black 3 & (6,3)\\
3\black 7 &(3,7)\\
\hline
6\black 7&(6,3,7)&\\
\hline\hline
\end{tabular}
\end{minipage}
\begin{minipage}{0.41\textwidth}
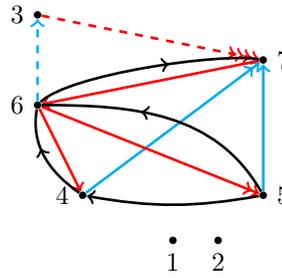
\begin{tikzpicture}[scale=.6]

\node[A,label=left:{ $4$}](2) at (1,0) {};
\node[A,label=below:{ $1$}](1) at (3,-1) {}; 

\node[A,label=right:{ $5$}](5) at (5,0) {};

\node[A,label=left:{ $6$}](6) at (0,2) {};
\node[A,label=right:{ $7$}](4) at (5,3) {};

\node[A,label=left:{ $3$}](3) at (0,4) {};

\node[A,label=below:{ $2$}](7) at (4,-1) {};

\draw[line width=.1em,cyan,->>](2)--(4);
\draw[line width=.1em,red,->>>>,dashed](3)--(4);
\draw[line width=.1em,red,->](4)--(6);
\draw[line width=.1em,red,<-](2)--(6);
\draw[line width=.1em,cyan,->](5)--(4);
\draw[line width=.1em,red,->>](6)--(5);
\draw[line width=.1em,cyan,->,dashed](6)--(3);
\draw[line width=.1em,
postaction={decorate, decoration={markings, mark=at position 0.6 with {\arrow{>}}}}]  
(6) to[out=45, in=170, looseness=.45] (4);
\draw[line width=.1em,
postaction={decorate, decoration={markings, mark=at position 0.6 with {\arrow{>}}}}]  
(2) to[out=145, in=-100] (6);
\draw[line width=.1em,
postaction={decorate, decoration={markings, mark=at position 0.6 with {\arrow{>}}}}]  
(5) to[out=120, in=0] (6);
\draw[line width=.1em,->] (5) edge[out=-170, in=-10] (2);
\end{tikzpicture}
\end{minipage}\\


\begin{tabular}{l|lc}
&STAGE 4\\
\hline\hline
5\black 4 & (5,1,4)\\
6\black 4 &(6,4)\\
4\black 6 & (4,2,6)\\
4\black 7 & (4,7)\\
\hline
5\black{\scriptsize a}  6&(5,1,4,2,6)&\\
5\black{\scriptsize a} 7  &(5,1,4,7)&\\
6\black 6 &(6,4,2,6)\ =\ $\Gamma_4$&\\
6\black{\scriptsize a}  7 &(6,4,7)&\\
\hline\hline
\end{tabular}
\hskip2cm
\begin{minipage}{0.41\textwidth}
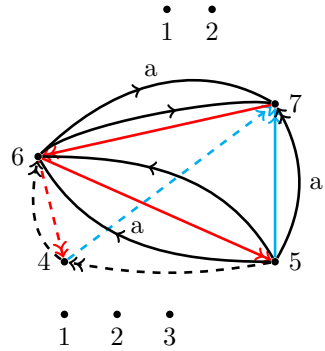
\begin{tikzpicture}[scale=.7]

\node[A,label=left:{ $4$}](2) at (1,0) {};
\node[A,label=below:{ $1$}](1) at (1,-1) {}; 

\node[A,label=right:{ $5$}](5) at (5,0) {};

\node[A,label=left:{ $6$}](6) at (0.5,2) {};
\node[A,label=right:{ $7$}](4) at (5,3) {};

\node[A,label=below:{ $3$}](3) at (3,-1) {};

\node[A,label=below:{ $2$}](7) at (2,-1) {};

\draw[line width=.1em,cyan,->>,dashed](2)--(4);

\draw[line width=.1em,red,->>](4)--(6);
\draw[line width=.1em,red,<<-,dashed](2)--(6);
\draw[line width=.1em,cyan,->>](5)--(4);
\draw[line width=.1em,red,->>](6)--(5);

\draw[line width=.1em,
postaction={decorate, decoration={markings, mark=at position 0.47 with {\arrow{>}}}}]  
 (6) to[out=45, in=150] node[midway, above] {a} (4);

\draw[line width=.1em,
postaction={decorate, decoration={markings, mark=at position 0.6 with {\arrow{>}}}}]  
(6) to[out=45, in=170, looseness=.45] (4);
\draw[line width=.1em,->>,dashed] (2) edge[out=145, in=-100] (6);
\draw[line width=.1em,
postaction={decorate, decoration={markings, mark=at position 0.6 with {\arrow{>}}}}]  
(5) to[out=180, in=-60] node[midway, above] {a} (6);
\draw[line width=.1em,
postaction={decorate, decoration={markings, mark=at position 0.6 with {\arrow{>}}}}]  
(5) to[out=120, in=0] (6);
\draw[line width=.1em,->>] (5) edge[out=60, in=-60]node [midway, right] {a}  (4);
\draw[line width=.1em,->>,dashed] (5) edge[out=-170, in=-10] (2);
\end{tikzpicture}
\end{minipage}
\vskip.25cm

\noindent As the stage graph changes parallel arcs may be introduced between the same end vertices. These parallel edges are distinguished by their retraction history, and by labels (like 5\new{\scriptsize a} 7).
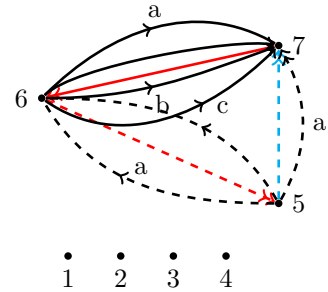
\begin{tabular}{l|lc}
&STAGE 5\\
6\black 5 & (6,5)&\\
5\new 6 &(5,1,2,6)\\
5\new{\scriptsize a} 6  &(5,1,4,2,6)\\
5\black  7 & (5,7)\\
5\new{\scriptsize a}   7&(5,1,4,7)\\
\hline
6\new 6& (6,5,1,2,6)\ =\ $ \Gamma_5$&\\
6\new{\scriptsize a} 6& (6,5,1,4,2,6)\ =\ $ \Gamma_6$&\\
6\new{\scriptsize b} 7& (6,5,7)&\\
6\new{\scriptsize c} 7& (6,5,1,4,7)&\\
\hline\hline
\end{tabular}
\hskip1.5cm
\begin{tikzpicture}[scale=.7]
\node() at (0,2){};
\node[A,label=below:{ $4$}](2) at (4,-1) {};
\node[A,label=below:{ $1$}](1) at (1,-1) {}; 

\node[A,label=right:{ $5$}](5) at (5,0) {};

\node[A,label=left:{ $6$}](6) at (0.5,2) {};
\node[A,label=right:{ $7$}](4) at (5,3) {};

\node[A,label=below:{ $3$}](3) at (3,-1) {};

\node[A,label=below:{ $2$}](7) at (2,-1) {};

\draw[line width=.1em,red,->>](4)--(6);

\draw[line width=.1em,cyan,->>,dashed](5)--(4);
\draw[line width=.1em,red,->>,dashed](6)--(5);

\draw[line width=.1em,->] (6) edge[out=45, in=170, looseness=.45] (4);

\draw[line width=.1em,->>,dashed] (5) edge[out=60, in=-60] node[midway, right] {a} (4);

\draw[line width=.1em,dashed,
postaction={decorate, decoration={markings, mark=at position 0.6 with {\arrow{>}}}}]  
(5) to[out=180, in=-60] node[midway, above] {a} (6);
\draw[line width=.1em,dashed,
postaction={decorate, decoration={markings, mark=at position 0.4 with {\arrow{>}}}}]  
(5) to[out=120, in=0] (6);

 \draw[line width=.1em, 
postaction={decorate, decoration={markings, mark=at position 0.5 with {\arrow{>}}}}]  
(6) to[out=45, in=150] node[midway, above] {a} (4);
\draw[line width=.1em,->,
postaction={decorate, decoration={markings, mark=at position 0.46 with {\arrow{>}}}}]  
 (6) to[out=-0, in=200]  node[midway, below] {b}(4);
\draw[line width=.1em,->,
postaction={decorate, decoration={markings, mark=at position 0.65 with {\arrow{>}}}}]  
 (6) to[out=-35, in=225] node[near end, below] {c} (4);
\end{tikzpicture}
\vskip.5cm
\begin{tabular}{l|ll}
&STAGE 6\\
\hline\hline
7\black 6 & (7,6)&\\
6\new 7 &(6,3,7)\\
6\new{\scriptsize a} 7 &(6,4,7))\\
6\new{\scriptsize b} 7 &(6,5,7)\\
6\new{\scriptsize c} 7 &(6,5,1,4,7)\\
\hline
7\new 7& (7,6,3,7)\ =\ $ \Gamma_7$&\\
7\new{\scriptsize a} 7&  (7,6,4,7)\ =\ $ \Gamma_8$&\\
7\new{\scriptsize b} 7&  (7,6,5,7)\ =\ $ \Gamma_9$\\
7\new{\scriptsize c} 7& (7,6,5,1,4,7)\ =\ $ \Gamma_{10}$&\\
\hline\hline
\end{tabular}
\hskip1cm
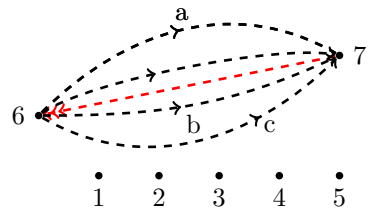
\begin{tikzpicture}[scale=.8]

\node[A,label=below:{ $4$}](2) at (4,1) {};
\node[A,label=below:{ $1$}](1) at (1,1) {}; 

\node[A,label=below:{ $5$}](5) at (5,1) {};

\node[A,label=left:{ $6$}](6) at (0,2) {};
\node[A,label=right:{ $7$}](4) at (5,3) {};

\node[A,label=below:{ $3$}](3) at (3,1) {};

\node[A,label=below:{ $2$}](7) at (2,1) {};

\draw[line width=.1em,red,->>>,dashed](4)--(6);
\draw[line width=.1em, dashed, 
postaction={decorate, decoration={markings, mark=at position 0.5 with {\arrow{>}}}}]  
(6) to[out=45, in=150] node[midway, above] {a} (4);

\draw[line width=.1em,->,dashed,
postaction={decorate, decoration={markings, mark=at position 0.4 with {\arrow{>}}}}]  
 (6) to[out=35, in=170, looseness=.5]  (4);

 \draw[line width=.1em, dashed, 
postaction={decorate, decoration={markings, mark=at position 0.5 with {\arrow{>}}}}]  
(6) to[out=45, in=150] node[midway, above] {a} (4);
\draw[line width=.1em,->,dashed,
postaction={decorate, decoration={markings, mark=at position 0.46 with {\arrow{>}}}}]  
 (6) to[out=-0, in=205]  node[midway, below] {b}(4);
\draw[line width=.1em,->,dashed,
postaction={decorate, decoration={markings, mark=at position 0.7 with {\arrow{>}}}}]  
 (6) to[out=-30, in=225] node[near end, below] {c} (4);
\end{tikzpicture}

\newpage

{\it The proof of  Claim 2.} Let $C\subset G_P$ be a directed cycle of length at least 2, and let $t_c>s_c$ be the largest and the second largest vertices of $C$. At the beginning of Stage $s_C$ every vertex of $C$ is eliminated except vertices $s_C$ and  $t_C$.  In a step of type (b) 
arcs $s_C \new t_C$ and $t_C \new s_C$ are combined to the 2-cycle $(s_C \new t_C \new s_C)$, thus obtaining a retracted form of $C$. Then the concatenation of the retraction histories of these arcs yields cycle $C\subset G_P$ and adds it to the list $\Gamma$.   
   We obtained that FME collects all cycles  of the \FaceGraph in   $\Gamma$  concluding the proof of Claim 2.\\
    
The cycle inequalities, that is, the sum of the arc weights  along a cycle, can be derived by applying Proposition \ref{CYCLE INEQUALITY RULE}. 
 Thus the Fourier-Motzkin elimination translates the  linear system (\ref{LRO}) to the equivalent system  (\ref{CYC}) for all cycles of  the key graph $G_P$. Therefore, a feasible solution of the  system (\ref{CYC}) can be extended with appropriate left endpoint values $\ell_1,\ldots,\ell_n$ to obtain a feasible solution of the  system (\ref{LRO}). 
\end{proof}

\newcommand{\vertexScale}{1.0} 
\newcommand{\edgeLabelScale}{1.0} 
\definecolor{edgeBlue}{RGB}{0,0,255}
\definecolor{edgeRed}{RGB}{255,0,0}
\definecolor{edgeBlack}{RGB}{0,0,0}
\def\XAxisBump{.7} 
\def\MaxVerticalLine{4}  
\def\MinVerticalLine{0}  
\def\IntervalThickness{1.4}
\def\Gap{1.8} 
\def\VGap{.5}   
\def\Epsilon{0.3}  
\def\Height{3} 

\section{A unique minimal Hilbert basis}
\label{Hilbertsection}

The {\em Hilbert basis} for a polyhedral cone, introduced by Giles and Pulleyblank \cite{GilesPulley}, is a finite set of vectors 
such that each integral vector in the cone
is a non-negative integral linear combination of vectors from this set.  A Hilbert basis
is {\em integral} if it consists entirely of integral vectors.  An old result of van der Corput (see \cite{Schrijver} page $233$) proves that a pointed rational cone has unique minimal integral Hilbert basis.  
Theorem \ref{Qcone}  proves that the length polyhedron, $Q_P$, of an interval order $P=(X,\prec)$
is a pointed affine cone. 
The apex of this conic polyhedron is $\CanonicalVector$, which
is an integral vector corresponding to the lengths of the intervals in the canonical representation of $P$.  
Thus $Q_P = \CanonicalVector + Q_P^*$,  where \LengthCone is a pointed rational cone.   
Since $Q_P^*$ has a minimal integral Hilbert basis, it must be unique by van der Corput's theorem. 
Our goal in this section is to determine this basis.
As we shall show (Theorem \ref{UniqueMinimalHilbertBasis}), the vectors in this integral Hilbert basis are characteristic vectors of non-empty subsets of elements of $X$, each 
corresponding to a minimal extendable collection of intervals.  

Intuitively, a collection of intervals of the canonical representation, called {\it canonical intervals},  is an {\em extendable collection of intervals} if
the intervals in this collection 
can be lengthened arbitrarily and simultaneously by the same amount 
without affecting the lengths of other intervals or compromising the interval representation.  
This restates the notion of a ray of the length polyhedron.
To be more precise and combinatorial, we now introduce some notation to relate this to the canonical representation of the interval order $P$.

For $i \in \{0,\ldots,m\}$ (where $m$ is the magnitude of $P$), define the {\em $i$th gap set}, denoted $\GapSet_i$, to be the elements of $P$ whose
canonical intervals contain the entire interval $[i-1,i]$; more precisely,
for $1\leq i\leq m-1$,
$$\GapSet_i  = \{x \in X: \ell_x\leq i-1 \mbox{ and } i \leq r_x\}$$
is the $i$-th gap set  also called
the $(i-1,i)\mbox{-gap}$, and $\GapSet_{0}  = \varnothing = \GapSet_{m}$.  The gaps defining the empty gap sets $\GapSet_{0}$ and $\GapSet_{m}$ are occasionally referred to as
the $(-\infty,0)\mbox{-gap}$ and $(m-1,\infty)\mbox{-gap}$, respectively.

Next we define the left and right borders of the $i$th gap; these are sets of 
elements whose canonical intervals touch the $i$th gap but do not cross the gap.  More precisely,  
the {\em $i$th left-borderline set}, denoted $\LeftBorderlineSet_i$,
and the {\em $i$th right-borderline set}, denoted $\RightBorderlineSet_i$, are defined to be
$\LeftBorderlineSet_i  = \{x \in X: r_x=i-1\}$ and 
$\RightBorderlineSet_i  = \{x \in X: \ell_x=i\}$, respectively.
Note that $\LeftBorderlineSet_{0} = \varnothing=
\RightBorderlineSet_{m}$.

 Now we are ready to define fundamental  
 extenders for a poset $P=(X, \prec)$.  A non-empty set $S \subseteq X$
 is a {\em  fundamental extender} if there exists an index $0 \leq i \leq m$ and a subset $Z \subseteq X$ such
 that $S =  \GapSet_{i} \cup Z$ and  $Z \subseteq \LeftBorderlineSet_i $ or $Z \subseteq \RightBorderlineSet_i$.  
 Note that a fundamental extender can be defined by distinct gaps, that is, 
 more than one index and borderline subset may witness 
 that it is a fundamental extender (see Fig.\ref{AHilbertBasisExample} for such examples).  Let  $\mathcal{E}_P$ be the collection of all fundamental extenders for $P$.
 
 Let $\overrightarrow{\mathcal{E}_P}$ denote the set of characteristic vectors of sets in $\mathcal{E}_P$.
 
 \begin{theorem}  \label{HilbertBasis} For any interval order $P$, $\overrightarrow{\mathcal{E}_P}$ is an integral Hilbert basis for the cone $Q_P^*$.
 \end{theorem}
 \begin{proof}  The vectors in  $\overrightarrow{\mathcal{E}_P}$  are characteristic vectors which are  integral.  It suffices to show that every integral vector in $Q_P^*$
 	is a non-negative integral combination of vectors from $\overrightarrow{\mathcal{E}_P}$.  To this end,
 	consider an arbitrary integral vector $\vec{v}\in Q_P^*$.  Now $\vec{v}$ corresponds to a natural 
 	interval representation $\mathcal{R}$ of $P$ in which all intervals have non-negative integer endpoints.  
 	If $\mathcal{R}$  is not the canonical representation, then
 	Theorem \ref{canon} implies there is some interval endpoint that is either not the left endpoint or not the right endpoint
 	of any interval.  By symmetry, we may suppose that $i$ is an interval endpoint in $\mathcal{R}$ that is not
 	the right endpoint of any interval.  Let $j$ be the next larger (left or right) endpoint of any interval, 
 	which must exist because there is at least one right endpoint larger than $i$.  Both $i$ and $j$ are integers, 
 	the open interval $(i,j)$ contains no endpoints, thus $[i,j]$ is a gap in $\mathcal{R}$, and no interval has a right endpoint at $i$.  

 	Lemma \ref{collapse} shows how this interval representation can be collapsed to a new interval representation $\mathcal{R}^\prime$ of $P$, 
 	essentially by contracting the interval $[i,j]$ to $[i]$.
 	Fig.\ref{ExtendingCollapsing} shades the interval $(i,j)$ because it contains no endpoints.  The eight possible
 	types of intervals (A-H) are depicted along with the resulting intervals produced upon contraction.
	
 	Firstly $\mathcal{R}^\prime$ corresponds to an integral vector in $Q_P^*$, say $\vec{v}^\prime$.
 	 Secondly, the length of intervals of type A, F, G, H does not change; furthermore, 
 	the set of elements $S$ whose intervals shorten by $j-i$ during this collapsing operation (types B-E) is a subset of the union of the entire $(h,i)$-gap in $\mathcal{R}^\prime$ and the 
 	right borderline set at $i$, where $h$ is the closest endpoint less than $i$ (see Fig.\ref{ExtendingCollapsing}).
 	Indeed sets of this type remain sets of this type upon performing such a contraction.
 	Consequently, repeated contractions to the canonical representation implies that
 	 $S$ is a subset of a union 
 	of a gap set and a
 	right borderline subset in the canonical representation; that is, $S$ is a
 	 fundamental extender.    Hence, $\vec{v}^\prime = \vec{v} - (j-i) \vec{u}$, where
 	$u \in \overrightarrow{\mathcal{E}_P}$ is the characteristic vector of $S$.   Thus repeated contractions leading to the canonical representation eventually show that 
 	$\vec{v}$ is a non-negative integral combination of vectors from $\overrightarrow{\mathcal{E}_P}$.
 \end{proof}
 
Theorem \ref{HilbertBasis} (with van der Corput's Theorem) implies that $\overrightarrow{\mathcal{E}_P}$ contains
a unique minimal Hilbert basis for $Q^*_P$.
We next turn to the task of describing this unique minimal Hilbert basis.
This reduces to determining which vectors
in $\overrightarrow{\mathcal{E}_P}$ are non-negative integral linear combinations of others.
We rephrase this condition in terms of the corresponding sets.
Let us call a subset $S \subseteq X$ a {\em Hilbert set} if $S \in \mathcal{E}_P$ and its characteristic vector is in the unique minimal Hilbert basis for $Q^*_P$.

\begin{theorem} \label{UniqueMinimalHilbertBasis}   
	 A fundamental extender $S$ is a Hilbert set if and only if $S$ is
	 not the disjoint union of proper fundamental extenders in $S$.
\end{theorem}
\begin{proof}  
	If $S$ the disjoint union of Hilbert sets, then $S$ is not a Hilbert set because its characteristic vector
	is a  non-negative integral linear combinations of characteristic vectors of the Hilbert sets in this disjoint union.
	On the other hand, consider an arbitrary  
	non-empty set $S\in \mathcal{E}_P$ that is not a Hilbert set. We must prove 
	that $S$ is the disjoint union of Hilbert sets.
	Because $S$ is not a Hilbert set, the characteristic vector of $S$, $\chi_S$, must be a
	non-negative integral linear combination of vectors from Hilbert sets, say 
	$$\chi_S = \sum_{\theta \in \Theta } \lambda_\theta \chi_{\theta},$$
	for some collection $\Theta$ of Hilbert sets and some set of positive integers $\{\lambda_\theta\}_{\theta \in \Theta }$.
	Because $\chi_S$ is a binary vector, it follows that $\lambda_\theta = 1$,
	for all $\theta \in \Theta$.  This implies the sets in $\Theta$ are disjoint.
	Therefore, $S$ is equal to the disjoint union $\bigcup_{\theta \in \Theta} \theta$.
\end{proof}

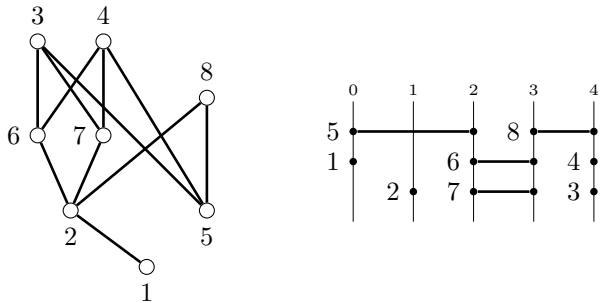
\begin{figure}[H]
	\begin{center}
	\begin{tikzpicture}[scale=0.25]
			\tikzset{vertex/.style={shape=circle, draw, inner sep = 2pt, minimum size = 5}}
			\tikzset{edge/.style = {}}
			\commentTikz{vertex 1}   \node[vertex,label=below:$1$] [fill = white,text = black, scale=\vertexScale]  (1) at (5.8,-2){};			
			\commentTikz{vertex 2}   \node[vertex,label=below:$2$] [fill = white,text = black, scale=\vertexScale]  (2) at (1.755,1){};
			\commentTikz{vertex 6}   \node[vertex,label=left:$6$] [fill = white,text = black, scale=\vertexScale]  (6) at (0,5){};
			\commentTikz{vertex 7}   \node[vertex,label=left:$7$] [fill = white,text = black, scale=\vertexScale]  (7) at (3.5,5){};
			\commentTikz{vertex 3}   \node[vertex,label=above:$3$] [fill = white,text = black, scale=\vertexScale]  (3) at (0,10){};			
			\commentTikz{vertex 4}   \node[vertex,label=above:$4$] [fill = white,text = black, scale=\vertexScale]  (4) at (3.5,10){};
			\commentTikz{vertex 8}   \node[vertex,label=above:$8$] [fill = white,text = black, scale=\vertexScale]  (8) at (9,7){};
			\commentTikz{vertex 5}   \node[vertex,label=below:$5$] [fill = white,text = black, scale=\vertexScale]  (5) at (9,1){};
			\commentTikz{arc 1 to 2}      \draw[edge] [edgeBlack, line width=.1em] 
			(1) to (2);
			\commentTikz{arc 5 to 3}      \draw[edge] [edgeBlack, line width=.1em] 
			(5) to (3);
			\commentTikz{arc 5 to 8}      \draw[edge] [edgeBlack, line width=.1em] 
			(5) to (8);
			\commentTikz{arc 2 to 8}      \draw[edge] [edgeBlack, line width=.1em] 
			(2) to (8);
			\commentTikz{arc 2 to 6}      \draw[edge] [edgeBlack, line width=.1em] 
			(2) to (6);
			\commentTikz{arc 3 to 7}      \draw[edge] [edgeBlack, line width=.1em] 
			(3) to (7);
			\commentTikz{arc 3 to 6}      \draw[edge] [edgeBlack, line width=.1em] 
			(3) to (6);
			\commentTikz{arc 4 to 5}      \draw[edge] [edgeBlack, line width=.1em] 
			(4) to (5);
			\commentTikz{arc 4 to 7}      \draw[edge] [edgeBlack, line width=.1em] 
			(4) to (7);
			\commentTikz{arc 6 to 4}      \draw[edge] [edgeBlack, line width=.1em] 
			(6) to (4);
			\commentTikz{arc 2 to 7}      \draw[edge] [edgeBlack, line width=.1em] 
			(2) to (7);
		\end{tikzpicture}
		\begin{minipage}{.1\textwidth}
		\end{minipage}\hskip1.15cm
		\begin{minipage}{.45\textwidth}
	\begin{tikzpicture}[scale=.8]
	\node () at (7,-1){};
\foreach \i in {0,...,4}{
    \draw[line width=.02em](3+\i,4)--(3+\i,6);
 \node[]()at(3+\i,6.2){\tiny{\i}};   }

\node[A,label=left:1](L) at (3,5) {};

\node[A,label=left:6](L) at (5,5) {};
\node[A](R) at (6,5) {};
\draw[line width=.1em](L)--(R); 

\node[A,label=left:7](L) at (5,4.5) {};
\node[A](R) at (6,4.5) {};
\draw[line width=.1em](L)--(R); 

\node[A,label=left:5](L) at (3,5.5) {};
\node[A](R) at (5,5.5) {};
\draw[line width=.1em](L)--(R); 

\node[A,label=left:8](L) at (6,5.5) {};
\node[A](R) at (7,5.5) {};
\draw[line width=.1em](L)--(R); 

\node[A,label=left:2]()at (4,4.5) {};
\node[A,label=left:4]()at (7,5) {};
\node[A,label=left:3]()at (7,4.5) {};
\end{tikzpicture}
		\end{minipage}
	\end{center}
\definecolor{edgeBlue}{RGB}{0,0,255}
\definecolor{edgeRed}{RGB}{255,0,0}
\vskip-3cm
\caption{\label{DPExample1} Hasse diagram and canonical representation of P=[0,1,2,2,0,2,2,3] (Example 1 from Doignon and Pauwels \cite{DoignonPauwels}).}
\label{Example1Doignon}
\end{figure}

\begin{figure}[H]
	\begin{center}
		\begin{longtable}{l|c|l|l}
			\hline \Tstrut
			\textbf{\#} & \textbf{Hilbert sets} & \textbf{gap set(s)} & \textbf{gap $\cup$ borderline subset} \\
			\hline
			
			\endfirsthead
			\multicolumn{4}{c}%
			{\tablename\ \thetable\ -- \textit{rays continued from previous page}} \\
			\hline \Tstrut
			\textbf{\#} & \textbf{Hilbert sets} & \textbf{gap set(s)} & \textbf{gap $\cup$ borderline subset} \\
			\hline
			\endhead
			\endfoot
			\hline
			\endlastfoot
			\Tstrut
			1 & \{1\} &  $ \begin{array}{c} (-\infty,0)\mbox{-gap} = \varnothing \\ \end{array}$ &  $ \begin{array}{c}\varnothing \cup \{1\} \\ \end{array}$  \Bstrut \\ \hline \Tstrut
			2 & \{3\} &  $ \begin{array}{c} (4,\infty)\mbox{-gap} = \varnothing \\ \end{array}$ &  $ \begin{array}{c}\varnothing \cup \{3\} \\ \end{array}$  \Bstrut \\ \hline \Tstrut
			3 & \{4\} &  $ \begin{array}{c} (4,\infty)\mbox{-gap} = \varnothing \\ \end{array}$ &  $ \begin{array}{c}\varnothing \cup \{4\} \\ \end{array}$  \Bstrut \\ \hline \Tstrut
			4 & \{5\} &  $ \begin{array}{c} (-\infty,0)\mbox{-gap} = \varnothing \\(0,1)\mbox{-gap} = \{5\} \\(1,2)\mbox{-gap} = \{5\} \\ \end{array}$ &  $ \begin{array}{c}\varnothing \cup \{5\} \\\{5\} \cup \varnothing \\\{5\} \cup \varnothing \\ \end{array}$  \Bstrut \\ \hline \Tstrut
			5 & \{8\} &  $ \begin{array}{c} (3,4)\mbox{-gap} = \{8\} \\(4,\infty)\mbox{-gap} = \varnothing \\ \end{array}$ &  $ \begin{array}{c}\{8\} \cup \varnothing \\\varnothing \cup \{8\} \\ \end{array}$  \Bstrut \\ \hline \Tstrut
			6 & \{2, 5\} &  $ \begin{array}{c} (0,1)\mbox{-gap} = \{5\} \\(1,2)\mbox{-gap} = \{5\} \\ \end{array}$ &  $ \begin{array}{c}\{5\} \cup \{2\} \\\{5\} \cup \{2\} \\ \end{array}$  \Bstrut \\ \hline \Tstrut
			7 & \{5, 6\} &  $ \begin{array}{c} (1,2)\mbox{-gap} = \{5\} \\ \end{array}$ &  $ \begin{array}{c}\{5\} \cup \{6\} \\ \end{array}$  \Bstrut \\ \hline \Tstrut
			8 & \{5, 7\} &  $ \begin{array}{c} (1,2)\mbox{-gap} = \{5\} \\ \end{array}$ &  $ \begin{array}{c}\{5\} \cup \{7\} \\ \end{array}$  \Bstrut \\ \hline \Tstrut
			9 & \{6, 7\} &  $ \begin{array}{c} (2,3)\mbox{-gap} = \{6, 7\} \\ \end{array}$ &  $ \begin{array}{c}\{6, 7\} \cup \varnothing \\ \end{array}$  \Bstrut \\ \hline \Tstrut
			10 & \{6, 8\} &  $ \begin{array}{c} (3,4)\mbox{-gap} = \{8\} \\ \end{array}$ &  $ \begin{array}{c}\{8\} \cup \{6\} \\ \end{array}$  \Bstrut \\ \hline \Tstrut
			11 & \{7, 8\} &  $ \begin{array}{c} (3,4)\mbox{-gap} = \{8\} \\ \end{array}$ &  $ \begin{array}{c}\{8\} \cup \{7\} \\ \end{array}$ 
		\end{longtable}
	\end{center}
   \caption{\label{AHilbertBasisExample} The unique minimal Hilbert basis of $Q^*_P$ for P=[0,1,2,2,0,2,2,3]}
\end{figure}

As an example we  took Example 1 from Doignon and Pauwels \cite{DoignonPauwels},
 the interval order $P=[0, 1, 2, 2, 0, 2, 2, 3]$. The Hasse diagram of $P$  is shown on the left of Fig.\ref{Example1Doignon},  its  canonical representation is on the right.
Table \ref{AHilbertBasisExample}  shows the $11$ Hilbert sets whose characteristic vectors form
the unique minimal Hilbert basis for the cone $Q^*_P$.  
The figure also shows the representations of these Hilbert sets as unions of gap sets and borderline subsets;
these {\it gap decompositions} witness that these sets are fundamental extenders.
As we noticed earlier, some of the Hilbert sets have multiple representations of this type.

\subsection{How to compute the minimal Hilbert basis for $Q^*_P$}

In this subsection, we address the computational complexity of computing the unique minimal Hilbert basis for the cone $Q^*_P$
of an interval order $P$.  The general complexity of computing Hilbert bases and related questions has been studied extensively, 
often revealing intractable problems even under strict conditions.   

For example,
Henk and Weismantel \cite{HenkWeismantel} proved that, given a pointed cone, one can find in polynomial time
an element of the (unique minimal) Hilbert basis that belongs to the boundary of the cone.  However, deciding whether
the Hilbert basis contains an interior point of the cone is NP-hard.

Seb\H{o}'s survey \cite{SeboSurvey} highlights that,
given two vectors $a, h \in \mathbb{Z}^n$, deciding whether $h$ is in the Hilbert basis
of the cone $\{x : ax = 0, x\geq 0\}$ is coNP-complete (Theorem 5.1 of the survey).
Additionally, Pap \cite{Pap} proved that 
deciding whether or not a set of integer vectors forms a
Hilbert basis is co-NP-complete even if the set consists of binary vectors having at most
three ones.  Seb\H{o}'s survey mentions several tractable combinatorial cases where computing 
a Hilbert basis is possible in polynomial time. 
Perfect graphs play a key role in some of these tractable instances. 

In this subsection, we demonstrate how to leverage perfect graphs to compute the unique minimal Hilbert basis for the cone  
$Q^*_P$ (Theorem \ref{HilbertExtenderGraphIsPerfect}).  A major consequence (Corollary \ref{HilbertPolySolvable}) is that, for any interval order $P=(X, \prec)$, if $\mathcal{E}_P$ has cardinality bounded by a polynomial in $|X|$, then the unique minimal Hilbert basis for $Q^*_P$
can be computed
in polynomial time.  This occurs for interval orders of bounded width, for example.
On the other hand, there are interval orders for which the unique minimal Hilbert basis for the cone  
$Q^*_P$ is exponential in the size of $P$.  We present such examples in the last  subsection.

One could restate Theorem \ref{UniqueMinimalHilbertBasis} as follows: a set $S$ is not a Hilbert set if and only if some
proper Hilbert subsets of $S$ pack perfectly into $S$.
It is well known that such set packing problems can be encoded using weighted intersection graphs.  We now introduce
this approach and describe its connection with perfect graphs.

Define the {\em Hilbert extender graph} of the interval order $P$, denoted $H_P$, as the weighted graph 
where:
\begin{itemize}
	\item each vertex $v$ of $H_P$ represents a Hilbert set $S_v$,
	\item there is an edge between vertices  $u$ and $v$ if  and only if $S_u\cap S_v\neq\varnothing$
	\item the weight of each vertex $v$ is $|S_v|$.
\end{itemize}

To determine whether a set  is not a Hilbert set, we consider all subsets of $S$ that are Hilbert sets.
Among these Hilbert sets we seek a collection of pairwise disjoint ones whose union is equal to $S$.
For $S\subset X$, let $H_P(S)$ denote the subgraph of $H_P$  induced by
the vertices representing proper subsets of $S$. A collection of pairwise disjoint Hilbert sets
corresponds to an independent set in $H_P(S)$.

Thus to test whether a set $S$ is not a Hilbert set it suffices to determine whether $H_P(S)$ contains
an independent set with weight $|S|$, which is clearly an upper bound for the maximum possible weight of an independent set in $H_P(S)$. 
Therefore,  testing whether $S$ is not a Hilbert set is essentially a {\sc maximum weighted independent set (MWIS)} problem:
for a given weighted graph,
select a subset of vertices with the maximum possible total weight, ensuring that no two selected vertices are connected by an edge.
It is well known that {\sc MWIS} is NP-complete, in general.  However, for perfect graphs 
 {\sc MWIS} it can be solved in polynomial time applying efficient linear programming algorithms, as noted by Gr\"otschel, Lov\'asz, Schrijver \cite{GLS}.  This explains how we leverage perfect graphs here.

\subsection{Hilbert extender graphs are perfect}
 In this subsection
we prove that a Hilbert extender graph $H=H_P$ is a Berge graph, that is,  $H$ contains no odd holes and no odd antiholes. Then the perfection of $H$ follows by applying the strong perfect graph theorem due to Chudnovsky et al. 
\cite{Chudnovsky} stating that Berge graphs are perfect.\\ 

To prepare the proof, for every Hilbert set  $S_v$, we fix an arbitrary gap decomposition of the form $S_v=G_i\cup Z$, where $G_i$ is the $i$th gap set  and $Z$ is a borderline subset of the $(i-1,i)$-gap, $0\leq i\leq m$. Furthermore, we  introduce a natural ordering of gap sets, we write $G_i < G_j$ if and only of $i<j$. 

\begin{lemma}
\label{3path}
If $(a,v,b)$ is an induced 3-path in $H$, and $S_v=G_i\cup Z$, then $G_i\neq\varnothing$.
\end{lemma}

\begin{proof}
If $G_i=\varnothing$, then every single-element subset of $Z$ is a fundamental extender, and so no subset of $Z$ of size at least two is in the minimal Hilbert basis. Hence $|S_v|=|Z|=1$. Because $a$ and $v$ are adjacent, $S_a\cap S_v\neq\varnothing$, hence $S_a\cap S_v=Z$; similarly, $S_b\cap S_v=Z$.  Consequently, $ S_a\cap S_b\neq\varnothing$, contradicting that $ab$ is not an edge of $(a,v,b)$.
\end{proof}

The gap decomposition of the fundamental extenders yield a natural coloring, called
a {\it gap coloring},
of the vertices of $H$: if $S_v=G_i\cup Z$, then we color $v$ with $i$. Observe that 
if  $a$ and $b$ are colored with the same color $j$ and 
$G_j\neq\varnothing$, then $S_a\cap S_b\neq\varnothing$, thus $ab\in H$. In other words, for any subgraph $C\subset H$
such that the gap coloring of its vertices uses non-empty gap sets, the gap coloring
of $C$ is a proper coloring for $\overline{C}$, the complement of $C$. 

Observe that every vertex of a hole or an antihole  
is the midpoint of an induced 3-path; thus by Lemma \ref{3path}, no gaps are empty in the gap decomposition of their vertices.  Therefore, a gap coloring of the Hilbert extender graph $H$ is proper for the antihole and hole subgraphs of $H$.
\begin{proposition}
\label{2colors}
If $C\subseteq H$ is a $k$-hole, $k\geq 4$, then the gap coloring uses exactly two colors on $C$. In particular,  a Hilbert extender graph contains no  $k$-hole with $k\geq 5$.
\end{proposition}
\begin{proof}  
Let $C=(1,\ldots,k)$ be a $k$-hole in $H$ ($k\geq 4$).
As noted above, the gap coloring properly colors $\overline{C}$, and since  $\overline{C}$ has no independent set of size three, every color appears at most two times. If a color appears twice, then it appears at  consecutive vertices on $C$.

Consider consecutive vertices $a , a+1$ on $C$ (cyclically), and let $S_{a}=G_i\cup Z_i$ and $S_{a+1}=G_j\cup Z_j$, with $i\neq j$.  We may assume that $i<j$ (by selecting appropriate orientation of $C$). Observe that no vertex $b$ exists in $C$ with $S_b=G_h\cup Z_h$, where
 $i< h< j$. To see this, take a canonical interval $J\in S_a\cap S_{a+1}$; since $J$ contains the $(h-1,h)$-gap, we have $J\in G_h$ implying $J\in S_a\cap S_b\cap S_{a+1}$. Therefore, $(a, b, a+1)$ is a triangle in $C$,  contradicting that $C$ is a $k$-hole with $k\geq 4$.
 
Because there is no color $h$ between $i$ and $j$, and this argument applies cyclically, there can be at most two  distinct colors, each used at most twice, thus   $k\leq 4$  follows. 
\end{proof}  
\begin{proposition}
\label{antiholes} 
The Hilbert extender graph contains no odd antihole of size $n\geq 7$.
\end{proposition}

\begin{proof}  Let $H$ be a Hilbert extender graph, and let $A\subset H$ be an antihole of odd size $n\geq 7$. 
The complement $\overline{A}$ is an $n$-cycle $C=(1,2,\ldots,n)$ 
 (depicted with dotted lines in the diagrams).
Consider an arbitrary proper coloring of $C$ (not necessarily a gap coloring). Let $1$ and $2$ be colored red and cyan respectively, and let $L=(4,\ldots,{n-1})$. 

Assume first that there is a vertex $j\in L$ having a color different from red and cyan.
Then $1,2,j$ plus a neighbor of $j$ in $L$ 
 induce a 4-hole in $A$  colored with  more than two colors contradicting Proposition \ref{2colors} (see the left figure). 

 
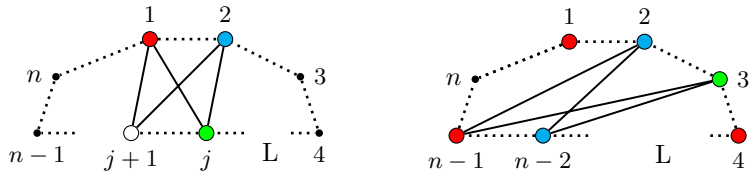
\begin{center}
\begin{tikzpicture}
\node[A,label=below:{\small${n-1}$}](a0)at (2,0){};

\node[X,label=below:{\small${j+1}$}](a1)at (3.25,0){};
\node[G,label=below:{\small${j}$}](b)at (4.25,0){};
\node[](a)at (5.25,0){};

 \node[A,label=left:{${n}$}](b1)at (2.25,.75 ){};
  \node[W,label=above:{\small${1}$}](a2)at (3.5,1.25){};
\node[B,label=above:{\small${2}$}](b2)at (4.5,1.25 ){};

\node[A,label=below:{\small${4}$}](b3)at (5.75,0){};
 \node[A,label=right:{\small${3}$}](a3)at (5.5,.75 ){};
 
\draw[line width=.1em,dotted](a)--(b3); 
    \draw[line width=.1em,dotted](a2)--(b2); 
      \draw[line width=.1em,dotted](b)--(a1)(2.5,0)--(a0)--(b1) ; 
      
    \draw[line width=.1em,dotted](a3)--(b3); 
     \draw[line width=.1em,dotted](4.75,0)--(b); 
           \draw[line width=.1em,dotted](a3)--(b2); 
               \draw[line width=.1em,dotted](a2)--(b1); 
               \draw[line width=.075em]  (b2)--(b)--(a2)--(a1)--(b2);(a1)--(a3)--(b);     
        \node()at (5.1,-.25){L};               
  \end{tikzpicture}
  \hskip1cm
  \begin{tikzpicture}
\node[W,label=below:{\small${n-1}$}](a1)at (2,0){};
\node[](a)at (5.25,0){};
\node[A,label=left:{${n}$}](b1)at (2.25,.75 ){};
 
  \node[W,label=above:{\small${1}$}](a2)at (3.5,1.25){};
\node[B,label=above:{\small${2}$}](b2)at (4.5,1.25 ){};
\node[W,label=below:{\small${4}$}](b3)at (5.75,0){};     
\node[G,label=right:{\small${3}$}](a3)at (5.5,.75 ){};
\node[B,label=below:{\small${n-2}$}](b)at (3.15,0){}; 
    \draw[line width=.1em,dotted](3.75,0)--(b); 
    
\draw[line width=.1em,dotted](a)--(b3); 
    \draw[line width=.1em,dotted](a2)--(b2); 
      \draw[line width=.1em,dotted](b)--(a1)--(b1);

     \draw[line width=.1em,dotted](a2)--(b1); 
           \draw[line width=.1em,dotted](b3)--(a3)--(b2); 
               \draw[line width=.1em,dotted](a2)--(b1); 
               \draw[line width=.075em](b2)--(b)(a1)--(b2)(a1)--(a3)--(b);     
         \node()at (4.75,-.25){L};                
  \end{tikzpicture}
\end{center}

We assume next that $L$ is properly colored with red and cyan. Because $L$ has an odd number of vertices, the end vertices $4$ and ${n-1}$ obtain the same color, say red. In this case
vertex $3$ needs a third color and $(2,{n-1},3,{n-2})$ is a three colored 4-hole
(see the right figure). Similarly, if both end vertices of $L$ are cyan then $(1,{4},n,5)$ is a  4-hole colored with three colors.
This contradicts Propositions \ref{2colors} and concludes the proof that $H$ has no odd antihole of size $n\geq 7$.
\end{proof}

\begin{theorem} \label{HilbertExtenderGraphIsPerfect} 
	The Hilbert extender graph of an interval order is 
	a perfect graph.
\end{theorem}
\begin{proof}
A Hilbert extender graph $H$ is a so called Berge graph, because it contains neither odd hole (Proposition \ref{2colors}) nor odd antihole (Proposition \ref{antiholes}).  Berge graphs are perfect due to the strong perfect graph theorem \cite{Chudnovsky}, thus $H$ is a Berge graph.
\end{proof}

Because 
induced subgraphs of a Hilbert extender graph $H=H_P$ are also perfect, 
testing whether a fundamental extender $S \in \mathcal{E}_P$ is a Hilbert set is polynomial-time solvable
provided that $H_P(S)$ is a graph whose order is bounded by a polynomial in $|X|$.

\begin{corollary} \label{HilbertPolySolvable} Suppose that $P=(X, \prec)$ is an interval order.  If
	$\mathcal{E}_P$ has cardinality bounded by a polynomial in $|X|$, then the unique minimal Hilbert basis for the pointed cone $Q^*_P$
	can be computed in polynomial time.
\end{corollary}
\begin{proof}  The Hilbert extender graph $H_P$ can be built in stages $k=1,2,\ldots $ with stage $k$ adding vertices with weight $k$.
Determining whether a vertex of weight $k$ is
added to $H_P$ only requires knowledge of the vertices with smaller weight as Theorem \ref{HilbertExtenderGraphIsPerfect} shows.
If $\mathcal{E}_P$ has cardinality bounded by a polynomial in $|X|$,
it follows that the number of vertices of weight $k$ is similarly bounded.
Consequently,
$H_P$ is a graph whose order is bounded by a polynomial in $|X|$ and it can be constructed  
in polynomial time.  This graph encodes the unique minimal Hilbert basis for the pointed cone $Q^*_P$, as desired.
\end{proof}

\definecolor{edgeBlack}{RGB}{0,0,0}

\begin{center}
\begin{figure}[htp]
	\begin{center}
		\begin{tikzpicture}[scale=0.3]
			\tikzset{vertex/.style={shape=circle, draw, inner sep = 2pt, minimum size = 5}}
			\tikzset{edge/.style = {}}
			\commentTikz{vertex 1}   \node[vertex,label=below:{\small$1$}] [fill = white,text = black, scale=\vertexScale]  (1) at (2.613636363000,0.001000000000){\normalsize{}};
			\commentTikz{vertex 2}   \node[vertex,label=left:{\small$2$}] [fill = white,text = black, scale=\vertexScale]  (2) at (-.3,5.000000000000){\normalsize{}};
			\commentTikz{vertex 3}   \node[vertex,label=left:{\small$3$}] [fill = white,text = black, scale=\vertexScale]  (3) at (2.500000000000,5.000000000000){\normalsize{}};
			\commentTikz{vertex 4}   \node[] [fill = white,text = black, scale=\vertexScale]  (4) at (4.999999999000,5.000000000000){\normalsize{$\cdots$}};
			\commentTikz{vertex 5}   \node[vertex,label=right:{\small$n-3$}] [fill = white,text = black, scale=\vertexScale]  (5) at (6.8,5.000000000000){};
			\commentTikz{vertex 6}   \node[vertex,label=right:{\small$n-2$}] [fill = white,text = black, scale=\vertexScale]  (6) at (11,5.000000000000){};
			\commentTikz{vertex 7}   \node[vertex,label=above: {\small$n-1$}] [fill = white,text = black, scale=\vertexScale]  (7) at (6.63636363000,10.000000000000){};
			\commentTikz{vertex 8}   \node[vertex,label=below:{$n$}] [fill = white,text = black, scale=\vertexScale]  (8) at (7.840909091000,0.001000000000){\normalsize{}};
			\commentTikz{arc 1 to 2}      \draw[edge] [edgeBlack, line width=.1em] 
			(1) to (2);
			\commentTikz{arc 1 to 3}      \draw[edge] [edgeBlack, line width=.1em] 
			(1) to (3);
			\commentTikz{arc 1 to 5}      \draw[edge] [edgeBlack, line width=.1em] 
			(1) to (5);
			\commentTikz{arc 1 to 6}      \draw[edge] [edgeBlack, line width=.1em] 
			(1) to (6);
			\commentTikz{arc 2 to 7}      \draw[edge] [edgeBlack, line width=.1em] 
			(2) to (7);
			\commentTikz{arc 3 to 7}      \draw[edge] [edgeBlack, line width=.1em] 
			(3) to (7);
			\commentTikz{arc 5 to 7}      \draw[edge] [edgeBlack, line width=.1em] 
			(5) to (7);
			\commentTikz{arc 6 to 7}      \draw[edge] [edgeBlack, line width=.1em] 
			(6) to (7);
		\end{tikzpicture}
	\end{center}
	\caption{\label{ExpHilbertBasis}  Hasse diagram of a family of interval orders with exponential-sized Hilbert basis.}
\end{figure}
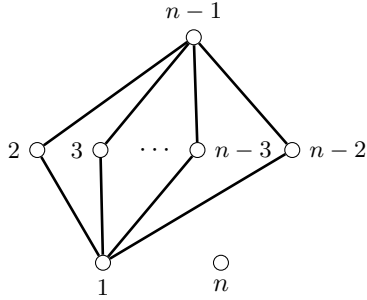 
\end{center}

\subsection{Exponential-sized Hilbert basis}

The number of vectors in the Hilbert basis may become
exponential in the size of the interval order as demonstrated by
the interval order whose Hasse diagram is shown in Figure \ref{ExpHilbertBasis}.
This interval order, call it $P$, has  
Hilbert basis for the cone $Q^*_P$ consisting of the $2^{n-3} + 2$ vectors corresponding to these Hilbert sets:
$\{1\}, \{n-1\}$ and $\{n\} \cup S$, for any subset  $S \subseteq \{2,\ldots,n-2\}$. 
This interval order has twins (pairs of incomparable elements that are comparable to the same sets of other elements).  It is worth noting that  twins in the example are not necessary to use;
it is easy to modify this construction to obtain a twin-free example
exhibiting an exponential-sized Hilbert basis.

\section{Conclusions}
\label{Conclusionsection}

Our motivation to study the length polyhedron of an interval order arises from an attempt to attack $k$-count problems, questions about interval representations using at most $k$ distinct interval lengths. 
A systematic study of the $k$-count representations for interval orders 
was initiated by Fishburn a half a century ago \cite{FishburnBook}. Since then several remarkable results were obtained by Isaak \cite{Isaak1993,Isaak2009} with several collaborators 
\cite{BoyadzhiyskaDissertation,Gordon,BGT}; and also by other groups with  
Oliviera and  Szwarcfiter \cite{FMOS,JLORS}. 
Nevertheless, 
the $k$-count problems are widely open in general.
Investigating the polyhedra attached to interval orders, in particular, focusing on the length polyhedron, may help understand the 
difficulty of these questions. 

The solution of the $k$-count problem for $k=1$ is widely known.
An interval order defined by a family of compact intervals of the same length is called a {\em semiorder}. The poset {\bf 3 + 1} has  four elements, $\{a,b,c,d\}$, and the 3-chain  $a\prec b\prec c$. 
Scott and Suppes \cite{ScottSuppes} obtained that   an interval order is a {\em semiorder} if and only if it contains  no  {\bf 3~+~1} as a subposet. The characterization of interval orders defined by a family of intervals of two lengths, short and tall, is not known.  Some progress is reported in this direction separately for  the 2-count problem of interval orders containing no {\bf 4~+~1} as a subposet \cite{BiKeLe}.

 We proposed in \cite{BiKeLe} the k-count problem  for the natural interval representation of permutations. 
The  depth of a permutation $\pi$ is defined as the length of the largest decreasing subsequence of $\pi$.  A representation  of $\pi$ is a family of intervals $ [\ell_j,r_j]$, $j=1,\ldots n$, with distinct  integral endpoints such that
 $
\ell_{\pi(1)}<\ldots< \ell_{\pi(n)}<r_1< \ldots<r_n$. 
The representation of a permutation is {$k$-count} if  the number of distinct interval lengths is $k$. Obviously, an interval representation of a depth-k permutation uses at least $k$ distinct interval lengths. 
We solved the case $k=2$ by proving the duality  that  depth-2 permutations have 2-count interval representations. Examples  show that no similar perfection remains true in general; there exist depth-3 permutations admitting no 3-count interval representation. In \cite{dodgy} there are initial results for the case $k=3$, addressing the problem of characterizing permutation having $3$-count interval representation.

The key graph of an interval order is a colored directed graph whose cycles describe the length polyhedron of the interval order. Several properties of these graphs are applied when determining the Schrijver system of the length polyhedron \cite{KeLe}. The information held by a key graph is equivalent with the information contained by the canonical representation of an interval order. In that sense a key graph is the link between the interval order and its length polyhedron. Thus recognizing key graphs seems to be a relevant question at the crossing of interval systems and polyhedral combinatorics. A possible characterization is presented in \cite{keygraph}.

  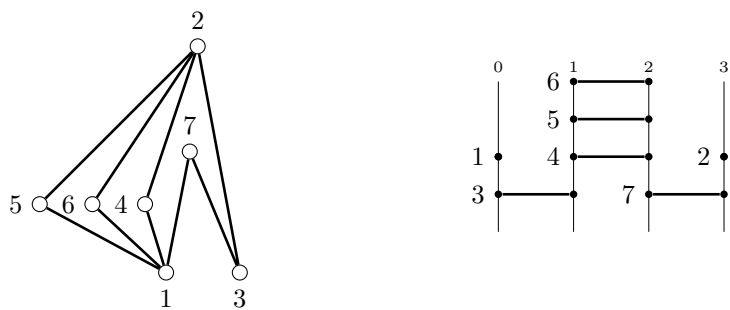
\begin{figure}[htp]
\begin{center}
\begin{tikzpicture}[scale=.7] 

\node[X,label=left:$4$](1) at (.5,3) {};
\node[X,label=left:$6$](3) at (-.5,3) {};
\node[X,label=left:$5$](4) at (-1.5,3) {};

\node[X,label=below:$3$](x) at (2.3,1.7) {};
\node[X,label=below:$1$](y) at (0.9,1.7) {};
\node[X,label=above:$2$](z) at (1.5,6) {};

\node[X,label=above:$7$](2) at (1.35,4) {};

\draw[line width=.1em](x)--(z)--(3)--(y)-- (2)--(x) (4)--(y)--(1);
\draw[line width=.1em](1)--(z)--(4);
\end{tikzpicture}
\hskip2cm
\begin{tikzpicture}
\node()at(0,-1){};
\foreach \i in {0,...,3}{
    \draw[line width=.01em](1+\i,0)--(1+\i,2);\node()at(1+\i,2.2) {\tiny\i};
}
\node[A,label=left:4](L) at (2,1) {};\node[A](R) at (3,1) {};
\draw[line width=.1em](L)--(R);
\node[A,label=left:2](L) at (4,1) {};\node[A](R) at (4,1) {};
\draw[line width=.1em](L)--(R);
\node[A,label=left:7](L) at (3,.5) {};\node[A](R) at (4,.5) {};
\draw[line width=.1em](L)--(R);
\node[A,label=left:6](L) at (2,2) {};\node[A](R) at (3,2) {};
\draw[line width=.1em](L)--(R);
\node[A,label=left:3](L) at (1,.5) {};\node[A](R) at (2,.5) {};
\draw[line width=.1em](L)--(R);
\node[A,label=left:1](L) at (1,1) {};\node[A](R) at (1,1) {};
\draw[line width=.1em](L)--(R);
\node[A,label=left:5](L) at (2,1.5) {};\node[A](R) at (3,1.5) {};
\draw[line width=.1em](L)--(R);
\end{tikzpicture}
\end{center}
\caption{The interval order $P=[0,1,0,1,1,1,2]$ has a Hilbert extender graph that is not weakly chordal}
\label{ex1}
\end{figure}
Theorem \ref{HilbertExtenderGraphIsPerfect} establishes that Hilbert extender graphs belong to the family of perfect graphs. 
However, it remains unclear whether this is a subfamily of any known small family of perfect graphs. 
Evidence suggests it is not, as there are Hilbert extender graphs that are not weakly 
chordal\footnote{\ Graphs containing no hole and no antihole of size more than 4 are called {\it weakly chordal}.}.
For instance, the interval order $P  = [0,1,0,1,1,1,2]$ in Fig.\ref{ex1} is not weakly chordal.
In its Hilbert extender graph
(of order $17$ and size $78$) the intersection of the Hilbert sets 
 $\{3,4\}$, $\{3,5\}$, $\{3,6\}$, $\{7,4\}$, $\{7,5\}$, and $\{7,6\}$,  is the complement of a $6$-cycle (that is, an antihole of size $6$).
  
A graph is {\em short-chorded} if every odd cycle has 
a chord joining two vertices of distance two in the cycle.
We  conjecture that Hilbert extender graphs are short-chorded. 
The characterization of Hilbert extender graphs is a challenging problem, and it would be also interesting to see how these structures are explaining further properties of interval families.

In this paper, we proved that the length polyhedron has a Hilbert basis composed of binary rays,
and the corresponding sets form an intersection graph known as the Hilbert extender graph, which is a perfect graph.
The fact that the Hilbert extender graph is a perfect graph is crucial for our polynomial-time algorithm that tests whether
a binary ray belongs to the Hilbert basis. An intriguing open question for future research is to characterize polyhedra
with binary Hilbert bases whose intersection graphs are perfect graphs.


\ifsidma\bibliographystyle{siamplain}\else\bibliographystyle{amsplain}\fi
\bibliography{MasterReferences}

\end{document}